\documentclass[11pt]{article}

\usepackage{amsthm}
\usepackage{txfonts}

\usepackage{float}
\usepackage{tikz} 
\usepackage{caption}
\usepackage{subcaption}
\usetikzlibrary{automata,positioning}
\usetikzlibrary{shapes}
\usepackage{graphicx}
\usepackage{caption}
\usepackage{graphics}
\usepackage{graphicx}
\usepackage{tabularx}
\usepackage{array}
\usepackage{epstopdf}
\usepackage{lipsum}
\usepackage{tablefootnote}
\usepackage[english]{babel}
\usepackage[utf8]{inputenc}
\usepackage{graphicx}
\theoremstyle{plain}
\usepackage{enumerate}
\usepackage{mathtools}
\usepackage{geometry}
\newtheorem{theorem}{Theorem}[section]
\newtheorem{lemma}[theorem]{Lemma}
\newtheorem{assumption}[theorem]{Assumption}

\newtheorem{prop}{Proposition}

\theoremstyle{remark}
\newtheorem{remark}{Remark}

\newcommand{\uv}{u_{\mathtt{V}}}
\newcommand{\up}{u_{\mathtt{P}}}

\newcommand{\xr}{x_{\mathtt{R}}}
\newcommand{\xs}{x_{\mathtt{S}}}
\newcommand{\xI}{x_{\mathtt{I}}}
\newcommand{\cP}{c_{\mathtt{P}}}
\newcommand{\cV}{c_{\mathtt{V}}}
\newcommand{\cI}{c_{\mathtt{I}}}

\newcommand{\gv}{\textbf{g}_{\mathtt{V}}}
\newcommand{\gp}{\textbf{g}_{\mathtt{P}}}

\newcommand\um[1]{{\color{black}  #1}}
\newcommand\rev[1]{{\color{black}  #1}}

\DeclareMathOperator*{\argmin}{arg\,min}
\usepackage{authblk}

\allowdisplaybreaks

\title{Optimal protection and vaccination against epidemics with reinfection risk}

\author[1]{Urmee Maitra\thanks{\texttt{urmeemaitra93@kgpian.iitkgp.ac.in}}}
\author[1]{Ashish R. Hota\thanks{\texttt{ahota@ee.iitkgp.ac.in}}}
\author[2]{Rohit Gupta\thanks{\texttt{rohit@aero.iitb.ac.in}}}
\author[3]{Alfred O. Hero\thanks{\texttt{hero@eecs.umich.edu}}}

\affil[1]{Department of Electrical Engineering, Indian Institute of Technology, Kharagpur}
\affil[2]{Department of Aerospace Engineering, Indian Institute of Technology, Bombay}
\affil[3]{Department of Electrical Engineering and Computer Science, University of Michigan, Ann Arbor}
\date{}
\begin{document}
\maketitle

\begin{abstract}
We consider the problem of the optimal allocation of vaccination and protection measures for the Susceptible-Infected-Recovered-Infected (SIRI) epidemiological model, which generalizes the classical Susceptible-Infected-Recovered (SIR) and Susceptible-Infected-Susceptible (SIS) epidemiological models by allowing for reinfection. First, we introduce the controlled SIRI dynamical model, and discuss the existence and stability of the equilibrium points. Then, we formulate a finite-horizon optimal control problem where the cost of vaccination and protection is proportional to the mass of the population that adopts it. Our main contribution in this work arises from a detailed investigation into the existence/non-existence of singular control inputs, and establishing optimality of bang-bang controls. The optimality of bang-bang control is established by solving an optimal control problem with a running cost that is linear with respect to the input variables. The input variables are associated with actions including the vaccination and imposition of protective measures (e.g., masking or isolation). In contrast to most prior works, we rigorously establish the non-existence of singular controls (i.e., the optimality of bang-bang control for our SIRI model). Under the assumption that the reinfection rate exceeds the first-time infection rate, we characterize the structure of both the optimal control inputs, and establish that the vaccination control input admits a bang-bang structure. The numerical results provide valuable insights into the evolution of the disease spread under optimal control.
\end{abstract}


\maketitle

\section{Introduction}\label{sec1}

 As observed during the COVID-19 pandemic, if left unchecked, infectious diseases potentially spread across the entire planet in the span of a few weeks and cause significant damage in terms of mortality and life-long impairments. In addition, the emergence of different variants may lead to reinfection, once the initial immunity weakens over time. \rev{In such situations, policy makers impose restrictions on individuals in the form of social distancing and mandatory mask usage. Additionally, they administer vaccines which offer partial immunity against the disease.} However, such interventions have a significant social and economic cost, and it is important to strike the right balance among the different options that are available. In this regard, dynamical systems and the optimal control theory have emerged as promising tools that provide policy-makers with appropriate guidelines and insights into mitigating epidemics (see, e.g., \cite{authour1,authour2}). Additionally, the optimal control of fractional-order systems has been used for spreading processes \cite{authour3}.

Starting from seminal works such as \cite{authour4,authour5}, there have been numerous investigations on optimal control of epidemiological processes, which largely consider compartmental dynamical models of epidemic evolution \cite{authour1}; see \cite{authour2} for a recent review. A majority of the past efforts have been \um{directed} towards optimal protection in the context of Susceptible-Infected-Recovered (SIR) epidemics and its variants (see, e.g., \cite{authour6,authour7,authour8,authour9}). More recent papers \cite{authour10,authour11,authour12,authour13} have considered vaccinations as additional control input (in addition to protection or social distancing measures). These works assume that the running cost is quadratic in the control input. However, it is natural to assume that the cost (of vaccination or protection) is proportional to the magnitude of the control input or the \um{fraction} of the population on which the input is administered. A few additional papers \um{(see, e.g., \cite{authour14,authour15}) have investigated the use of optimal control techniques,} when the population size is dynamically changing. Other related approaches are also explored in \cite{authour16,authour17,authour18}. The methodology for containing COVID-19 Delta strain spread was explored in \cite{authour19}. The authors included asymptomatic agents and captured the notion of an imperfect vaccination in their model. It is well established that fractional order optimal controls have advantages in the form of greater flexibility and higher accuracy over the classical integer order controls. The fractional order models of COVID-19 and other diseases were thoroughly explored in \cite{authour20, authour21, authour22}.

There have been limited investigations into epidemiological models where recovery does not give permanent immunity, and hence, reinfection is \um{also} possible as a result. In addition, even past works that assumed a cost functional that was linear in the control input which led to a bang-bang optimal control structure, the possibility of the existence of singular arcs and singular control inputs is not often examined in a rigorous manner (the work \cite{authour23} is a notable exception in this regard). Nevertheless, in practice, it is important to characterize the possibility of singular control inputs in order to provide insights into policy-making decisions, thus informing the authorities of the expected impact of either imposing or relaxing interventions.

The motivation for this work is to establish the existence of non-singular optimal policies to control the spread of epidemics via limited vaccination and protective measures by solving an optimal control problem considering a running cost that is linear with respect to the input variables in an epidemic model with reinfection risk. Our setting differs from most prior studies on the optimal control of epidemics that assume the objective function to be superlinear in the control inputs, which leads to a simpler analysis, and the issue of singular inputs can be avoided. For example, in \cite{authour24}, the authors used the SIR model where the objective function was quadratic in the control inputs. However, it is more reasonable to consider the running cost to be linear with respect to the input variables; indeed, the cost of vaccination (and other protection measures) is directly proportional to the fraction of the population being vaccinated (or adopting protective measures). While some studies, such as \cite{authour25}, assumed running costs that were linear in the control inputs, they focused on bang-bang controls without ruling out the possibility of singular controls. In this work, we consider a generalized epidemiological~model that incorporates both recovery and reinfection (similar to observations made during COVID-19), specifically the susceptible-infected-recovered-infected (SIRI) epidemic model (see e.g., \cite{authour26}). Our model includes both non-pharmaceutical and medical resources as inputs, and the running cost is assumed to be linear in these control variables. Additionally, during COVID-19, we observed higher reinfection rates due to variants such as Delta and Omicron, which supports the focus on compromised immunity in this work. Under appropriate assumptions, we specifically exclude the possibility of simultaneous singularities and analytically prove that vaccination control does not exhibit singularities.

 In the SIRI model, the rate of reinfection is different from the rate of initial infection, with higher values indicating compromised immunity and a smaller rate of reinfection indicating a partial immunity imparted by the disease and/or vaccinations. As analyzed in \cite{authour27}, it was assumed that vaccinations were only available for the susceptible sub-population who transit to the recovered compartment, thus reflecting the fact that vaccinations impart a certain degree of protection for the short term, but not complete immunity (a similar phenomenon was also observed during the COVID-19 pandemic).

The main contributions of the paper are as follows. We analyze the optimality conditions for the associated optimal control problem and rule out the existence of a simultaneous singularity of both control inputs on the SIRI epidemiological model \um{in the scenario of compromised immunity}. Then, we carry out a detailed investigation regarding the singularity of the vaccination control input, and \um{under sufficient conditions, we} show that it does not admit a singular arc (i.e., the vaccination-optimal control is always at one of two possible extreme admissible values). A theoretical analysis that provides valuable insight on the vaccination-control input being non-singular (also known as bang-bang or, on-off control) is essential, since bang-bang control is often considered a more appropriate intervention in practical epidemiological and clinical settings (see, e.g., \cite{authour28, authour29}). Additionally, we demonstrate epidemic evolution under optimal control inputs for \um{a} numerical case study and show the relative impact of vaccination and protection in an epidemic containment.

The remainder of the paper is organized as follows: the controlled SIRI epidemiological model is introduced in Section \ref{sec2}, where we also prove the existence and local asymptotic stability of its equilibrium points; the optimal control problem is presented in Section \ref{sec3} and the structural properties of the optimal control inputs are established; the non-existence of singular arcs in the structure of the \um{candidate} optimal control corresponding to vaccination is established in Section \ref{sec5}; numerical results depicting the evolution of the epidemic under the optimal control inputs are presented in Section \ref{sec:numerical}; and we conclude in Section \ref{sec6} with a discussion on possible directions for future research. 

%%%%%%%%%%%%%%%%%%%%%%%%%%%%%%%%%%%%%%%%%%%%%%%%%%%%%%%%%%%%%%%%%%%%%%%%%%%%%
%%%%%%%%%%%%%%%%%%%%%%%%%%%%%%%%%%%%%%%%%%%%%%%%%%%%%%%%%%%%%%%%%%%%%%%%%%%%%
%%%%%%%%%%%%%%%%%%%%%%%%%%%%%%%%%%%%%%%%%%%%%%%%%%%%%%%%%%%%%%%%%%%%%%%%%%%%%

\section{Controlled SIRI epidemiological model}\label{sec2}
Motivated by the COVID-19 pandemic, we consider the SIRI epidemiological model, which was introduced in \cite{authour26}. In this setting, an individual remains in one of three possible states: susceptible $(\mathtt{S})$, infected $(\mathtt{I})$, or recovered $(\mathtt{R})$. However, recovery is not permanent, and recovered individuals also become potentially infected again upon contact with infected individuals. The rate at which a susceptible (respectively, recovered) individual becomes infected upon contact with infected individuals is denoted by $\beta > 0$ (respectively, $\hat{\beta} > 0$). In general, $\beta$ is assumed to be different from $\hat{\beta}$. When $\hat{\beta} < \beta$, the reinfection rate is smaller than the rate of new infection, which indicates that recovery imparts a partial immunity against future infection. Similarly, $\hat{\beta} > \beta$ indicates a compromised immunity following the initial infection. Finally, $\gamma > 0$ represents the rate at which an infected individual recovers, which is referred to as recovery rate. The various state transitions are depicted in Figure \ref{fig:siri_tran}.
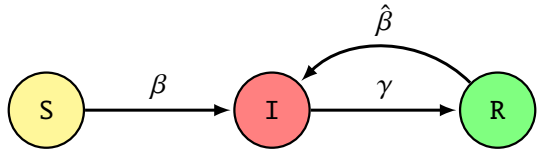
\begin{figure}[H]
\centering
\begin{tikzpicture}
% Setup the style for the states
\tikzset{node style/.style={state, minimum width=1cm, line width=0.3mm, fill=yellow!50!white}}
\tikzset{node style1/.style={state, minimum width=1cm, line width=0.3mm, fill=red!50!white}}
\tikzset{node style2/.style={state, minimum width=1cm, line width=0.3mm, fill=green!50!white}}
        % Draw the states
\node[node style] at (0, 0) (St)     {$\mathtt{S}$};
\node[node style1] at (3, 0) (At)     {$\mathtt{I}$};
\node[node style2] at (6, 0) (Rt)  {$\mathtt{R}$};
        % Connect the states with arrows
\draw[every loop, auto=right,line width=0.4mm,
              >=latex,
              draw=black,
              fill=black]
            %(St)     edge[bend right=20]            node {0.1} (Yt)
(St) edge[bend right=0, auto=left] node[above] {$\beta$} (At)
(At) edge[bend right=0] node[above] {$\gamma$} (Rt)
(Rt) edge[bend right=45] node[above] {$\hat{\beta}$} (At);
\end{tikzpicture}
\caption{Evolution of the states in the SIRI epidemic model (self-loops are omitted for better clarity).}
\label{fig:siri_tran}
\end{figure}

We consider two types of control inputs (which are assumed to be essentially bounded Lebesgue measurable functions): $u_{\texttt{V}}$, which captures the rate at which susceptible individuals are vaccinated, and $\up$, where $1 - \up$ captures the effective rate of social distancing or protection adoption by individuals in the disease states $\texttt{S}$ and $\texttt{R}$. As a consequence of the above control inputs, the resulting controlled SIRI epidemic dynamical equations are given by the following:
\begin{align}
\label{system_eq}
\begin{rcases}
\dot{x}_{\texttt{S}}(t) =-\beta x_{\texttt{S}}(t) x_{\texttt{I}}(t) \up(t) - \xs(t) \uv(t),
\\ \dot{x}_{\texttt{I}}(t) = \beta x_{\texttt{S}}(t) x_{\texttt{I}}(t) \up(t) + \hat{\beta}x_{\texttt{R}}(t) x_{\texttt{I}}(t) \up(t) - \gamma x_{\texttt{I}}(t),
\\ \dot{x}_{\texttt{R}}(t) =-\hat{\beta} x_{\texttt{R}}(t) x_{\texttt{I}}(t) \up(t)  +  \xs(t) \uv(t) + \gamma x_{\texttt{I}}(t),
\end{rcases}
\end{align}
where the state variables $x_{\texttt{S}}(t) \in [0,1], x_{\texttt{I}}(t) \in [0,1]$ and $x_{\texttt{R}}(t) \in [0,1]$ denote the instantaneous proportion of individuals in each of the three epidemic states $\mathtt{S}, \mathtt{I}$, and $\mathtt{R}$. Henceforth, unless required, we suppress the dependency of the states and control inputs on time $t$ for an improved readability. 
\begin{remark}
    The biological significance to controlling the epidemic spread is that our analysis accounts for reinfection in individuals, which aligns with the characteristics of several infectious~diseases, such as COVID-19, that only confers a short-term immunity. In addition, during the COVID-19 pandemic, it was demonstrated that the reinfection rates, particularly due to variants such as Delta and Omicron, exceed the initial infection rates \cite{authour30}. Motivated by this observation, we later assume that $\hat{\beta} > \beta$, which implies that getting infected compromised the immunity. These characteristics are not captured by the classical epidemic models, such as the SIR model.
\end{remark}  

\begin{remark}
Note that the term $\beta \xs \xI \up$ represents the fraction of the susceptible population who do not adopt any protection and get infected, while the term $\xs \uv$ represents the fraction of the susceptible population who opt for vaccination and move to the recovered state. Note that such individuals may become infected again in future (i.e., vaccination does not impart a permanent immunity against future infection). Similarly, the term $\hat{\beta} x_{\texttt{R}} x_{\texttt{I}} \up$ captures the fraction of the recovered population who do not adopt any protection and get \um{reinfected}, and the term $\gamma x_{\texttt{I}}$ is the fraction of the infected population who naturally recover. The above dynamics satisfies $\dot{x}_{\texttt{S}}(t) + \dot{x}_{\texttt{I}}(t) + \dot{x}_{\texttt{R}}(t) = 0$ for almost every instant of time $t$, and since the states represent \um{fractions} of the population, they also satisfy $x_{\texttt{S}}(t) + x_{\texttt{I}}(t) + x_{\texttt{R}}(t) = 1$ for every instant of time $t$, when the initial state vector also satisfies this condition (for details, see Lemma \ref{lem:invariant-set}). 
\end{remark}

In our model, the control inputs include behavioral measures (represented by $1 - u_{\mathtt{P}}$), such as protective behaviors or social distancing, and medical interventions (represented by $u_{\mathtt{V}}$), such as vaccination. Thus, our model accounts for both medical and non-medical interventions available during an epidemic. We impose limitations on both types of inputs to prevent trivial solutions that might arise from an unlimited supply of protection and vaccination. In the above setting, $\up = 0$ implies that the susceptible or recovered individuals adopt complete a protection and they do not bear the risk of getting infected. In order to rule out this impractical corner case, we assume that $\up$ is always bounded from below by a lower bound $u_{\texttt{Pmin}}>0$. \um{Additionally, we assume that $u_{\texttt{P}} \leq 1$ with the upper bound chosen to signify that $\beta$ and $\hat{\beta}$ denote the infection rates in the absence of any protective action.} In addition, we assume that the vaccination rate satisfies $0 \leq u_{\texttt{V}} \leq u_{\texttt{Vmax}} < 1$, where we have limited the upper threshold by excluding~$1$, as $\uv = u_{\texttt{Vmax}} = 1$ would imply vaccinating the entire susceptible fraction of population in one go, which is not practical. 

\begin{remark}
When $\hat{\beta} = 0$ (i.e., recovered individuals do not get \um{reinfected}), then the model reduces to the SIR epidemiological model \um{(see \cite{authour31})}. Similarly, as mentioned in \cite{authour31}, when $\beta = \hat{\beta}$ (i.e., the infection rate of susceptible and recovered individuals coincide), then we recover the Susceptible-Infected-Susceptible (SIS) epidemiological model. Thus, the SIRI epidemiological model studied \um{in this paper} is a strict generalization of both the SIS and SIR \um{epidemiological models} (see, e.g., \cite{authour26}). 
\end{remark}

Before stating the optimal control problem studied in this paper, \um{we} first establish certain theoretical properties of the controlled SIRI epidemiological model when the control inputs are exogenous constants. First, we investigate the equilibria of \um{its dynamics} and their associated stability properties. \um{When $u_{\texttt{V}} = 0$, the dynamics in \eqref{system_eq} is an instance of the SIRI model without any explicit control input, with the infection rates effectively being $\beta \up$ and $\hat{\beta} \up$, respectively.} The equilibria and their stability properties follow from analogous results established for the classical SIRI \um{epidemiological model} in \cite{authour19}. Therefore, we focus on the case where the constant steady-state inputs are defined by $u_{\texttt{V}} = \uv^{\texttt{eq}}$, where $0 < \uv^{\texttt{eq}} \leq u_{\texttt{Vmax}}$, and $\up = \up^{\texttt{eq}}$, where $u_{\texttt{Pmin}} \leq \up^{\texttt{eq}} \leq 1$.

By equating the right hand side of \eqref{system_eq} to zero, we observe the existence of two equilibrium points:
\begin{enumerate}[(i)]
     \item The disease free equilibrium point $E_{\mathtt{DFE}}$ for  $(x^{\texttt{eq}}_\texttt{S} = 0, x^{\texttt{eq}}_\texttt{I} = 0, x^{\texttt{eq}}_\texttt{R} = 1)$, which always exists;
     \item The endemic equilibrium point $E_{\mathtt{EE}}$ for $(x^{\texttt{eq}}_\texttt{S} = 0, x^{\texttt{eq}}_\texttt{I} = 1 - \frac{\gamma}{\hat{\beta}  u^{\texttt{eq}}_\texttt{P}}, x^{\texttt{eq}}_\texttt{R} = \frac{\gamma}{\hat{\beta} u^{\texttt{eq}}_\texttt{P}})$, which exists when $\gamma < \hat{\beta}  u^{\texttt{eq}}_\texttt{P}$.
\end{enumerate}
%where, $(\cdot)^{\texttt{eq}}$ is the time-invarying quantity (state or control input) at equilibrium. 
The following result establishes their local stability properties. 

\begin{prop}[Local asymptotic stability of the equilibrium points]
For the controlled SIRI epidemiological model \eqref{system_eq} with \um{$u^{\textnormal{\texttt{eq}}}_\textnormal{\texttt{V}} > 0$}, we have the following: 
\begin{enumerate}[(i)]
    \item The disease free equilibrium point is locally asymptotically stable when $\gamma > \hat{\beta} \um{u^{\textnormal{\texttt{eq}}}_\textnormal{\texttt{P}}}$;
    \item The endemic equilibrium point is locally asymptotically stable when $\gamma < \hat{\beta} \um{u^{\textnormal{\texttt{eq}}}_\textnormal{\texttt{P}}}$.
\end{enumerate}
\label{prop:eq_pts}
\end{prop}

\begin{proof}
Since \um{$ x_{\texttt{S}}(t) + x_{\texttt{I}}(t) + x_{\texttt{R}}(t) = 1$ for every instant of time $t$} (see Lemma \ref{lem:invariant-set}), we equivalently consider the dynamics involving only the two state variables $x_\texttt{S}$ and $x_\texttt{I}$\um{, by} expressing $x_\texttt{R} = 1 - x_\texttt{S} - x_\texttt{I}$. Thus, the \um{dynamics} \eqref{system_eq} reduces to the following:
\begin{align*}
%\label{system_eq_new}
\dot{x}_{\texttt{S}}(t)& =-\beta x_{\texttt{S}}(t) x_{\texttt{I}}(t) \up(t) - \xs(t) \uv(t),
\\ \dot{x}_{\texttt{I}}(t) &= \beta x_{\texttt{S}}(t) x_{\texttt{I}}(t) \up(t) + \hat{\beta}(1 - \xs(t) - \xI(t)) x_{\texttt{I}}(t) \up(t) - \gamma x_{\texttt{I}}(t),
\end{align*}
the Jacobian matrix of which is given by the following:
\begin{equation*}
     J(\xs, \xI, \up, \uv) = \begin{bmatrix}
         -\beta x_\texttt{I} \up - u_\texttt{V} & -\beta x_\texttt{S} \up \\
        \;\;\; (\beta - \hat{\beta}) x_\texttt{I} \up \;\;\;\;\;\;\;\;\; & \beta x_\texttt{S} \up + \hat{\beta} (1 - x_\texttt{S} - 2 x_\texttt{I}) \up - \gamma
     \end{bmatrix}.
\end{equation*}
 First, we investigate the local asymptotic stability of the disease-free equilibrium point\um{. In this case, the Jacobian matrix} is given by the following:
 \begin{equation*}
     \um{J(E_{\mathtt{DFE}})} = \begin{bmatrix}
         - u^{\texttt{eq}}_\texttt{V} & 0 \\
         0 & \hat{\beta} \up^{\texttt{eq}} - \gamma
     \end{bmatrix}.
 \end{equation*} 
%The eigenvalues of the above matrix are $(-u^{\texttt{eq}}_\texttt{V}, \hat{\beta} \up^{\texttt{eq}} - \gamma)$. 
Since $u^{\texttt{eq}}_\texttt{V} > 0$\um{,} both the eigenvalues of the above matrix\um{, are strictly negative when} $\gamma > \hat{\beta} \up^{\texttt{eq}}$.\footnote{Intuitively, when $\gamma < \hat{\beta} u^{\texttt{eq}}_\texttt{P}$, the recovery rate of the infected fraction of population is less than the rate of reinfection of the recovered fraction of population times the fraction of recovered people not adopting protection, leading to the disease free equilibrium being unstable.} 
Next, we investigate the local asymptotic stability of the endemic equilibrium point. The Jacobian matrix in this case, is given by
 \begin{equation*}
     \um{J(E_{\mathtt{EE}})} = \begin{bmatrix}
        -\beta (1- \frac{\gamma}{\hat{\beta} \up^{\texttt{eq}}}) \up^{\texttt{eq}} - u^{\texttt{eq}}_\texttt{V} & 0 \\
         (\beta -\hat{\beta}) (1- \frac{\gamma}{\hat{\beta} \up^{\texttt{eq}}})\up^{\texttt{eq}} \;\;\;\;\;\; & -\hat{\beta} \up^{\texttt{eq}} + \gamma
     \end{bmatrix}.
 \end{equation*}  
It is easy to see that both the eigenvalues of the above matrix, are strictly negative when $\gamma < \hat{\beta} \up^{\texttt{eq}}$. This concludes the proof.
\end{proof}
%\vspace*{-1cm}
The following lemma supports the consideration of the SIRI epidemiological model.
\begin{lemma}[Positive invariant set of the controlled SIRI epidemiological model]
\label{lem:invariant-set}
The set $S \coloneqq \{(x_{\textnormal{\texttt{S}}},x_{\textnormal{\texttt{I}}},x_{\textnormal{\texttt{R}}}) : (x_{\textnormal{\texttt{S}}},x_{\textnormal{\texttt{I}}},x_{\textnormal{\texttt{R}}}) \in [0,1] \times [0,1] \times [0,1]\}$ is positively invariant, with respect to any unique global solution for the SIRI epidemic dynamics \eqref{system_eq}.
\end{lemma}

\begin{proof}
First, we show that local solutions exist and are also unique on a sufficiently small time-interval for the SIRI epidemic dynamics \eqref{system_eq}. To this end, fix a sufficiently small real number $\varepsilon>0$, and for any given Lebesgue measurable control inputs $\textbf{u} = (\up, \uv): [0, \varepsilon]\rightarrow [u_{{\texttt{Pmin}}},1]\times [0,u_{{\texttt{Vmax}}}]$, let us rewrite \eqref{system_eq} as follows:
\begin{equation*}
    \dot{\textbf{x}} = \textbf{F}(\textbf{x}, \textbf{u}(t)) \coloneqq \textbf{G}(t, \textbf{x}),~ \textbf{x}(0)=\textbf{x}_{0},
    \label{eq:compact_sys_eq}
\end{equation*}
where $\textbf{x} \coloneqq (\xs, \xI, \xr) \in \mathbb{R}^{3}$ is the state vector and the function $\textbf{G}: \mathbb{R} \times \mathbb{R}^3 \rightarrow \mathbb{R}^3$ is given by the following:
\begin{equation*}
\textbf{G}(t, \textbf{x}) \coloneqq \underbrace{\begin{bmatrix}
    0\\
    -\gamma \xI\\
    \gamma \xI
\end{bmatrix}}_{\textbf{f}(\textbf{x})} + \underbrace{\begin{bmatrix}
    -\beta \xs \xI\\
    \beta \xs \xI + \hat{\beta} \xr \xI\\
    -\hat{\beta} \xI \xr
\end{bmatrix}}_{\textbf{g}_{1}(\textbf{x})}\up(t) + \underbrace{\begin{bmatrix}
    -\xs\\
    0\\
    \xs
\end{bmatrix}}_{\textbf{g}_{2}(\textbf{x})}\uv(t).
\end{equation*}
Now, let $K \coloneqq \{(t, \textbf{x}) : 0 \leq t \leq \varepsilon, |\textbf{x} - \textbf{x}_{0}| \leq \varepsilon\}$ be a cylinder in $\mathbb{R} \times \mathbb{R}^{3}$.\footnote{\um{The norm of a vector $\textbf{x} \in \mathbb{R}^{3}$, is denoted by $|\textbf{x}|$.}} It is now easy to verify that the function $\textbf{G}$, satisfies the following conditions:
\begin{enumerate}
\item[(i)] {For almost every $t \in [0, \varepsilon]$, the mapping $\textbf{x}\mapsto \textbf{G}(t, \textbf{x})$ is continuous, and for every $\textbf{x} \in \mathrm{B}_{\varepsilon}(\textbf{x}_{0})$, the mapping $t\mapsto \textbf{G}(t, \textbf{x})$ is Lebesgue measurable;\footnote{The closed ball of radius $r>0$ in $\mathbb{R}^{3}$, centered at $\textbf{x} \in \mathbb{R}^{3}$, is denoted by $\mathrm{B}_{r}(\textbf{x})$.}}
\item[(ii)] {There exists a constant $C_{K}>0$ such that:
\begin{equation*}
    |\textbf{G}(t, \textbf{x})| \leq C_{K},
\end{equation*}
holds for almost every $t \in [0, \varepsilon]$ and for every $\textbf{x} \in \mathrm{B}_{\varepsilon}(\textbf{x}_{0})$. Moreover, there also exists a constant $L_{K}>0$ such that the following inequality:
\begin{equation*}
    |\textbf{G}(t, \textbf{x}) - \textbf{G}(t, \textbf{y})| \leq L_{K} |\textbf{x} - \textbf{y}|,
\end{equation*}
holds for almost every $t \in [0, \varepsilon]$ and for every $\textbf{x}, \textbf{y} \in \mathrm{B}_{\varepsilon}(\textbf{x}_{0})$.}
\end{enumerate}
Indeed, to verify the first claim in item (ii) stated above, one can obtain the following:
\begin{equation*}
    |\textbf{G}(t, \textbf{x})| \leq \sup_{\textbf{x} \in \mathrm{B}_{\varepsilon}(\textbf{x}_{0})} |\textbf{f}(\textbf{x})| + \sup_{\textbf{x} \in \mathrm{B}_{\varepsilon}(\textbf{x}_{0})} |\textbf{g}_{1}(\textbf{x})| + \sup_{\textbf{x} \in \mathrm{B}_{\varepsilon}(\textbf{x}_{0})} |\textbf{g}_{2}(\textbf{x})|u_{{\texttt{Vmax}}},
\end{equation*}
which holds for almost every $t \in [0, \varepsilon]$ and for every $\textbf{x} \in \mathrm{B}_{\varepsilon}(\textbf{x}_{0})$. Keeping in mind, the fact that the functions $\textbf{f}, \textbf{g}_{1}, \textbf{g}_{2}$ are of class $C^{1}$, one can now obtain the desired result by invoking Weierstrass' theorem. To verify the second claim in item (ii) stated above, one can obtain the following inequality:
\begin{equation*}
    |\textbf{G}(t, \textbf{x})-\textbf{G}(t, \textbf{y})| \leq |\textbf{f}(\textbf{x}) - \textbf{f}(\textbf{y})| + |\textbf{g}_{1}(\textbf{x}) - \textbf{g}_{1}(\textbf{y})| + |\textbf{g}_{2}(\textbf{x}) - \textbf{g}_{2}(\textbf{y})|u_{{\texttt{Vmax}}},
\end{equation*}
which holds for almost every $t \in [0, \varepsilon]$ and for every $\textbf{x}, \textbf{y} \in \mathrm{B}_{\varepsilon}(\textbf{x}_{0})$. Keeping in mind, the facts that the functions $\textbf{f}, \textbf{g}_{1}, \textbf{g}_{2}$ are of class $C^{1}$ and also that the set $\mathrm{B}_{\varepsilon}(\textbf{x}_{0})$ is compact and convex, one can now obtain the desired result. By appropriately modifying some of the steps given in the proof of \cite[Theorem~2.2.1]{authour32}, one can now deduce that local solutions exist for \eqref{system_eq} on the time-interval $[0, \epsilon]$. In addition, from \cite[Theorem 2.1.3]{authour32}, one can also deduce the uniqueness of such local solutions of \eqref{system_eq} on the time-interval $[0, \hat \epsilon]$, where $0 < \hat \varepsilon \leq \varepsilon$.

Next, we verify the positive invariance of the set $S$ with respect to the unique local solution $\textbf{x} = (\xs, \xI, \xr): [0, \hat \varepsilon] \rightarrow \mathbb{R}^3$ of the SIRI epidemic dynamics \eqref{system_eq}. To this end, from \eqref{system_eq}, we have that
\begin{align}
\label{eq:x=}
\begin{rcases}
\xs(t) = \exp \left(-\displaystyle \int_{0}^{t} \Big(\beta \xI(\tau) \up(\tau) + \uv(\tau) \Big) \, d\tau\right) \xs(0),\\
\xI(t) = \exp \left(\displaystyle \int_{0}^{t} \Big(\big(\beta \xs(\tau) + \hat{\beta} \xr(\tau)\big)\up(\tau) - \gamma \Big) \, d\tau \right) \xI(0),\\
\xr(t) = \exp \left(\displaystyle \int_{0}^{t} \left(-\hat{\beta} \xI(\tau) \up(\tau) \right) \, d\tau \right) \bigg[\xr(0)\\
\quad \quad\quad + \displaystyle \int_{0}^{t}  \exp \left(\displaystyle \int_{0}^{\tau} \left(\hat{\beta} \xI(s) \up(s) \right) \, ds \right) \Big(\xs(\tau) \uv(\tau) + \gamma \xI(\tau)\Big) \, d\tau \bigg],
\end{rcases}
\end{align}
for all $t \in [0, \hat \varepsilon]$. Now, observe that the initial conditions $\xs(0), \xI(0), \textnormal{ and } \xr(0)$ represent the initial fractions of the population who are susceptible, infected, and recovered, respectively, and therefore satisfy the following constraints: $\xs(0), \xI(0), \xr(0) \geq 0$ and $\xs(0) + \xI(0) + \xr(0) = 1$. Since, $\xs(0), \xI(0)\geq 0$, it follows from the first two equations in \eqref{eq:x=} that $\xs(t), \xI(t)\geq 0$ for all $t \in [0, \hat \varepsilon]$, and since $\xr(0) \geq 0$, $\uv(t)\geq 0$ for all $t \in [0, \hat \varepsilon]$ and $\gamma \geq 0$, it now follows from the third equation in \eqref{eq:x=} that $\xr(t)\geq 0$ for all $t \in [0, \hat \varepsilon]$. Now, from the dynamics \eqref{system_eq}, it follows that the relation $\dot{x}_{\mathtt{S}}(t) + \dot{x}_{\mathtt{I}}(t) + \dot{x}_{\mathtt{R}}(t) = 0$ is satisfied for almost every $t \in [0, \hat \varepsilon]$, which in turn implies that the relation $\xs(t) + \xI(t) + \xr(t)= 1$ is satisfied for all $t \in [0, \hat \varepsilon]$. Overall, the unique local solution $\textbf{x} = (\xs, \xI, \xr): [0, \hat \varepsilon] \rightarrow \mathbb{R}^3$ of the SIRI epidemic dynamics \eqref{system_eq}, satisfies the following constraints: $\xs(t), \xI(t), \xr(t) \geq 0$ and $\xs(t) + \xI(t) + \xr(t) = 1$ for all $t \in [0, \hat \varepsilon]$, from which it follows that $(\xs(t), \xI(t), \xr(t)) \in [0,1] \times [0,1] \times [0,1]$ for all $t \in [0, \hat \varepsilon]$.

Finally, we show that any unique right-maximal solution $\textbf{x} = (\xs, \xI, \xr): [0, T) \rightarrow \mathbb{R}^3$ of the SIRI epidemic dynamics \eqref{system_eq}, where the time $T > 0$, can be globally extended (i.e., it is possible to show that this holds for $T = \infty$). To this end, let us assume that $T < \infty$; then, by appropriately modifying some of the steps given in the proof of \cite[Theorem 2.1.4]{authour32}, one can deduce the following relation:
\begin{equation}
    \lim_{t \uparrow T} \left(|\textbf{x}(t)| + \frac{1}{d((t, \textbf{x}(t)), \partial \Omega)} \right) = \infty,
    \label{eq:gcc}
\end{equation}
where the set $\Omega \coloneqq [0, T) \times \mathbb{R}^{3}$ is an open set in $[0, \infty] \times \mathbb{R}^{3} \subset \mathbb{R} \times \mathbb{R}^{3}$, the notation $\partial \Omega$ denotes its boundary (in the set $[0, \infty] \times \mathbb{R}^{3}$), and the distance of a pair of points $(t, \textbf{y}) \in [0, \infty] \times \mathbb{R}^{3}$ to a set $K \subseteq [0, \infty] \times \mathbb{R}^{3}$ is given by $d((t, \textbf{y}), K) \coloneqq \inf_{(\bar t, \bar {\textbf{y}}) \in K} |(t, \textbf{y}) - (\bar t, \bar {\textbf{y}})|$. Using the facts that $S$ is a compact set in $\mathbb{R}^{3}$ and is positively invariant with respect to the unique right-maximal solution $\textbf{x} = (\xs, \xI, \xr): [0, T) \rightarrow \mathbb{R}^3$ of the SIRI epidemic dynamics \eqref{system_eq}, together with the fact that the boundary $\partial \Omega = (\{T\} \times \mathbb{R}^{3}) \cup ([0, T] \times \emptyset)$, we now arrive at a contradiction in view of \eqref{eq:gcc}. This completes the proof.
\end{proof}
\begin{remark}
    Note that the positive invariance of the set $S$ in the proof of Lemma \ref{lem:invariant-set} can also be shown using Nagumo's theorem adapted to control systems (see, e.g., \cite[Theorem $4.11$]{authour33}).
\end{remark}
\raggedbottom

\section{Optimal control problem}\label{sec3}
Now, we consider the following optimal control problem:
\begin{equation}
\label{eq:ocp_original}
\begin{rcases}
\begin{aligned}
    \inf_{\substack{\up(\cdot)\in L^{\infty}([0,T];\mathbb{R}) \\ \uv(\cdot)\in L^{\infty}([0,T];\mathbb{R})}} &\int_{0}^{T} \um{[}c_{\texttt{P}} (1-\up(t)) (x_{\texttt{S}}(t) + x_{\texttt{R}}(t))+c_{\texttt{V}} u_{{\texttt{V}}}(t) x_{\texttt{S}}(t)+c_{{\texttt{I}}}x_{\texttt{I}}(t)\um{]} \,dt,\\\
    \textrm{s.t.} \hspace{26pt} &\dot{x}_{\texttt{S}}(t) =-\beta x_{\texttt{S}}(t) x_{\texttt{I}}(t) \up(t) - \xs(t) \uv(t),
    \\& \dot{x}_{\texttt{I}}(t) = \beta x_{\texttt{S}}(t) x_{\texttt{I}}(t) \up(t) + \hat{\beta}x_{\texttt{R}}(t) x_{\texttt{I}}(t) \up(t) - \gamma x_{\texttt{I}}(t), 
    \\& \dot{x}_{\texttt{R}}(t) =-\hat{\beta} x_{\texttt{R}}(t) x_{\texttt{I}}(t) \up(t)  +  \xs(t) \uv(t) + \gamma x_{\texttt{I}}(t),
    \\& (\xs(0), \xI(0), \xr(0)) \in [0, 1] \times [0, 1] \times [0, 1],
    \\& u_{{\texttt{Pmin}}} \leq \up(t) \leq 1~\text{for a.e.}~t\in [0,T],
    \\& 0 \leq \uv(t) \leq u_{{\texttt{Vmax}}}~\text{for a.e.}~t\in [0,T],
    \end{aligned}
\end{rcases}
\end{equation}
where $u_{{\texttt{Pmin}}} > 0$ represents the minimum fraction of the susceptible or recovered sub-populations who remain unprotected, and $u_{{\texttt{Vmax}}} < 1$ denotes an upper bound on the \um{fraction} of the \um{total} population that can be vaccinated for a given time period. The individual weighing terms in the running cost \eqref{eq:ocp_original} are as follows: $\cP$ captures the cost incurred due to the protection adopted by the susceptible and recovered individuals, $\cV$ captures the cost incurred due to vaccination by the susceptible \um{individuals}, and $\cI$ is the disease cost or the cost incurred on being infected.

\subsection{Equivalent formulation in Mayer form}
In order to exploit the results from the optimal control theory in our subsequent analysis, we convert the optimal control problem defined by \eqref{eq:ocp_original} into the Mayer form. In the Mayer form, the cost functional solely consists of the terminal cost. This requires appending an additional state $x_\texttt{C}$ to our dynamics which captures the running cost, and whose time-evolution satisfies the following dynamics:
\begin{equation}
    \dot{x}_\texttt{C}(t)=c_{\texttt{P}} (1-\up(t)) (x_{\texttt{S}}(t) + x_{\texttt{R}}(t)) +c_{\texttt{V}} u_{{\texttt{V}}}(t) x_{\texttt{S}}(t) + c_{{\texttt{I}}}x_{\texttt{I}}(t),~{x}_\texttt{C}(0)=0,
    \label{eq:xc}
\end{equation}
for almost every time instant $t$. In addition, we note that any one of the three epidemic states can be~expressed in terms of the other two states since the SIRI dynamics \eqref{system_eq} satisfies $x_\texttt{S}(t)+x_\texttt{I}(t)+x_\texttt{R}(t)=1$ for every instant of time $t$ and $\dot{x}_\texttt{S}(t)+\dot{x}_\texttt{I}(t)+\dot{x}_\texttt{R}(t)=0$ for almost every instant of time $t$. As a result, the epidemic dynamics can be expressed in terms of only two state variables. We express $x_\texttt{I}=1-x_\texttt{S}-x_\texttt{R}$ and omit the variable $x_\texttt{I}$ from the epidemic dynamics. By introducing the state vector $\textbf{z}=({x}_\texttt{C}, x_\texttt{S}, x_\texttt{R}) \in \mathbb{R}^{3}$, the optimal control problem defined by \eqref{eq:ocp_original} can now be written in the Mayer form as follows:
\begin{equation}
\label{eq:ocp}
\begin{rcases}
\begin{aligned}
    \inf_{\substack{\up(\cdot)\in L^{\infty}([0,T];\mathbb{R}) \\ \uv(\cdot)\in L^{\infty}([0,T];\mathbb{R})}} &{x}_\texttt{C}(T) \\\
    \textrm{s.t.} \hspace{26pt} &\dot{\textbf{z}}=\textbf{f}(\textbf{z})+\gp(\textbf{z}) \up +\gv(\textbf{z}) \uv,
    \\& \textbf{z}(0)\in \{0\}\times [0, 1] \times [0, 1],
    \\& u_{{\texttt{Pmin}}} \leq \up \leq 1~ \textrm{ for a.e.}~ t \in [0, T],
    \\& 0 \leq \uv \leq u_{{\texttt{Vmax}}}~ \textrm{ for a.e.}~ t \in [0, T],
    \end{aligned}
\end{rcases}
\end{equation}
where the drift and control vector fields are given by the following:
\begin{equation}
    \textbf{f}(\textbf{z})=\begin{bmatrix}
   \!\!\!\!\!\!\!\!&\cP (x_\texttt{S} + x_\texttt{R}) +c_\texttt{I}(1-x_\texttt{S}-x_\texttt{R})\\
    \!\!\!\!\!\!\!\!&0\\
    \!\!\!\!\!\!\!\!&\gamma (1-x_\texttt{S}-x_\texttt{R})
\end{bmatrix},~
\gp(\textbf{z})=\begin{bmatrix}
    \!\!\!\!\!\!\!\!&-\cP (x_\texttt{S} + \xr) \\
    \!\!\!\!\!\!\!\!&-\beta x_\texttt{S} (1-x_\texttt{S}-x_\texttt{R})\\
    \!\!\!\!\!\!\!\!&-\hat{\beta} \xr (1-x_\texttt{S}-x_\texttt{R}) \nonumber
\end{bmatrix},~
\gv(\textbf{z})=\begin{bmatrix}
    \!\!\!\!\!\!\!\!&c_\texttt{V} x_\texttt{S}\\
    \!\!\!\!\!\!\!\!&-x_\texttt{S}\\
    \!\!\!\!\!\!\!\!&x_\texttt{S}
\end{bmatrix}\um{.}
\label{eq:matrix2}
\end{equation}

%%%%%%%%%%%%%%%%%%%%%%%%%%%%%%%%%%%%%%%%%%%%%%%%%%%%%%%%%%%%%%%%%
%%%%%%%%%%%%%%%%%%%%%%%%%%%%%%%%%%%%%%%%%%%%%%%%%%%%%%%%%%%%%%%%%
%%%%%%%%%%%%%%%%%%%%%%%%%%%%%%%%%%%%%%%%%%%%%%%%%%%%%%%%%%%%%%%%%

\subsection{Existence of an optimal control input}
Now, we establish the existence of a solution for the optimal control problem defined by \eqref{eq:ocp}. To this end, we leverage Filippov’s theorem, which is stated below for the reader's convenience.
\begin{theorem}(Filippov’s theorem, \cite[Section 4.5]{authour35})
    \label{Filippov}
Consider a controlled dynamical system:
\begin{equation}
\label{eq:ft}
    \dot{\textbf{x}}(t) = \textbf{f}(t, \textbf{x}(t), \textbf{u}(t)),~\textbf{x}(0) = \textbf{x}_0,
\end{equation}
where $\textbf{x}(t) \in \mathbb{R}^{n}$ and $\textbf{u}(t) \in U \subset \mathbb{R}^{m}$. Assume that the solutions of \eqref{eq:ft} exist on a given time-interval $[0, T]$ for all control inputs $\textbf{u}(\cdot)$. In addition, assume that for every pair $(t, \textbf{x}) \in [0, T] \times \mathbb{R}^{n}$, the set $\{ f(t, \textbf{x}, \textbf{u}) : \textbf{u} \in U\}$ is compact and convex. Then, the reachable set $R^t(\textbf{x}_0)$ is compact for each $t \in [0, T]$.\footnote{The reachable set at time $t>0$, starting from $\textbf{x}_0 \in \mathbb{R}^{n}$, is defined as follows:
\begin{equation*}
    R^t(\textbf{x}_0) \coloneqq \{\textbf{x}(t) : \textbf{x}(\cdot; \textbf{x}_0, \textbf{u})~\text{is a solution of}~\eqref{eq:ft}~\text{defined over the time-interval $[0, t]$, corresponding to an admissible control input}~\textbf{u}(\cdot)\}.
\end{equation*}}
\end{theorem}

Now, we state the following theorem. 
\begin{theorem}[Existence of an optimal control]
There exists a solution for the optimal control problem defined by \eqref{eq:ocp}.
\end{theorem}
\begin{proof}
In view of Lemma \ref{lem:invariant-set}, it is clear that there exists a unique solution (with respect to any given initial condition and admissible control inputs) defined over the time-interval $[0, T]$ for the dynamics given in the optimal control problem defined by \eqref{eq:ocp}. Moreover, it is easy to verify that for every $\textbf{z} \in \mathbb{R}^{3}$, the set $\{\textbf{f}(\textbf{z})+\textbf{g}_{\texttt{P}}(\textbf{z}) u_{{\texttt{P}}} +\textbf{g}_{\texttt{V}}(\textbf{z}) u_{{\texttt{V}}} : (\up, \uv) \in [u_{{\texttt{Pmin}}}, 1] \times [0, u_{\texttt{Vmax}}]\}$ is a compact and convex set in $\mathbb{R}^{3}$. Now, it follows from Theorem \ref{Filippov} that the reachable set at time $T$, starting from any given initial condition, is a compact set in $\mathbb{R}^{3}$. Now, the proof can be concluded by invoking Weierstrass' extreme value theorem.
\end{proof}

%%%%%%%%%%%%%%%%%%%%%%%%%%%%%%%%%%%%%%%%%%%%%%%%%%%%%%%%%%%%%%%%%%%%%
%%%%%%%%%%%%%%%%%%%%%%%%%%%%%%%%%%%%%%%%%%%%%%%%%%%%%%%%%%%%%%%%%%%%%
%%%%%%%%%%%%%%%%%%%%%%%%%%%%%%%%%%%%%%%%%%%%%%%%%%%%%%%%%%%%%%%%%%%%%

\subsection{Structure of optimal control inputs}\label{sec:structure_OC}
Now, we establish the structure of the candidate optimal control inputs. To this end, we first use Pontryagin's maximum principle to single out the optimal control inputs. The Hamiltonian function which corresponds to the optimal control problem defined by \eqref{eq:ocp} is given by the following:
\begin{align}
    H(\textbf{z,u},\pmb{\lambda}) &=\langle \pmb\lambda,\textbf{f}(\textbf{z})+\textbf{g}_{\texttt{P}}(\textbf{z}) u_{{\texttt{P}}} +\textbf{g}_{\texttt{V}}(\textbf{z}) u_{{\texttt{V}}}\rangle \nonumber
    %\label{eq:H}
    \\ &=\lambda_\texttt{C} (\cP (\xs + \xr) + \cI (1 - \xs - \xr) - \cP (\xs + \xr) \up + \cV \xs \uv) \nonumber
    \\& \quad + \lambda_\texttt{S} (-\beta x_\texttt{S} (1-x_\texttt{S}-x_\texttt{R}) \up -x_\texttt{S} u_\texttt{V}) + \lambda_\texttt{R} (\gamma (1-x_\texttt{S}-x_\texttt{R})\nonumber
    \\&\quad -\hat{\beta} \xr (1-x_\texttt{S}-x_\texttt{R}) \up +x_\texttt{S}u_\texttt{v}),
    \label{eq:H_full}
\end{align}
where $\langle \cdot , \cdot \rangle$ denotes the standard inner-product of two vectors in $\mathbb{R}^3$ and, $\pmb\lambda = (\lambda_\texttt{C}, \lambda_\texttt{S}, \lambda_\texttt{R}) \in \mathbb{R}^{3}$ denotes the co-state vector.\footnote{\um{The absence of the case of the abnormal multiplier being equal to zero} for the optimal control problem defined by \eqref{eq:ocp} directly follows as a consequence of \cite[Corollary 22.3]{authour34}.}
For almost every time $t \in [0, T]$, the \um{minimizing control inputs are} given by the following:
\begin{equation*}
    \textbf{u}^*(t) = \argmin_{\textbf{u}\in [u_{{\texttt{Pmin}}},1]\times [0,u_{{\texttt{Vmax}}}]} H(\textbf{z}^*(t),\textbf{u},\pmb{\lambda}^*(t)),
    \label{eq:u_min}
\end{equation*}
where the superscript $(\cdot)^*$ denotes the optimal trajectories, and the co-state dynamics are given by
\begin{equation}
  \dot{\pmb\lambda}^*(t)=\begin{pmatrix}
-\frac{\partial H(\textbf{z}^*(t),\textbf{u}^*(t),\pmb{\lambda}^*(t))}{\partial x_\texttt{C}}, -\frac{\partial H(\textbf{z}^*(t),\textbf{u}^*(t),\pmb{\lambda}^*(t))}{\partial x_\texttt{S}}, -\frac{\partial H(\textbf{z}^*(t),\textbf{u}^*(t),\pmb{\lambda}^*(t))}{\partial x_\texttt{R}}\\
\end{pmatrix},
\label{eq:lambda_dot}
\end{equation}
which satisfies the following terminal boundary condition:
\begin{equation}
   \pmb\lambda^*(T)= (1, 0, 0).
    \label{transversality}
\end{equation}
From \eqref{eq:lambda_dot}, we obtain the following co-state dynamics:
\begin{align}
    \dot{\lambda}^*_\texttt{C}(t) &=0, \nonumber
    \label{lambda=0}
   \\ \dot{\lambda}^*_{\texttt{S}}(t) &=-\lambda^*_\texttt{C}(t) (c_\texttt{P}-c_\texttt{I}-c_\texttt{P} u^*_\texttt{P}(t) +c_\texttt{V} u^*_\texttt{V}(t))-\lambda^*_\texttt{S}(t)(-\beta (1-2 x^*_\texttt{S}(t) -x^*_\texttt{R}(t)) u^*_\texttt{P}(t)-u^*_\texttt{V}(t)) \nonumber
   \\ & \quad -\lambda^*_\texttt{R}(t)(-\gamma+\hat{\beta} x^*_\texttt{R}(t) u^*_\texttt{P}(t)+u^*_\texttt{V}(t)), 
   \\ \dot{\lambda}^*_{\texttt{R}}(t) &=-\lambda^*_\texttt{C}(t) (c_\texttt{P}-c_\texttt{I}-c_\texttt{P} u^*_\texttt{P}(t))-\lambda^*_\texttt{S}(t) \beta x^*_\texttt{S}(t) u^*_\texttt{P}(t) -\lambda^*_\texttt{R}(t) (-\gamma-\hat{\beta}(1-x^*_\texttt{S}(t)-2x^*_\texttt{R}(t)) u^*_\texttt{P}(t)). \nonumber
\end{align}
Since the Hamiltonian function is affine with respect to the control inputs, the structure of the optimal control inputs will be governed by the so-called switching functions given by the following (see, e.g.,~\cite[Section~4.4.3]{authour35}):
\begin{align}
    \phi_\texttt{P}(t) &=\langle \pmb\lambda^*(t),\textbf{g}_\texttt{P}(\textbf{z}^*(t))\rangle \nonumber
    \\ &= -\lambda^*_\texttt{C}(t) \cP (\xs^*(t) + \xr^*(t))-(\lambda^*_\texttt{S}(t) \beta x^*_\texttt{S}(t) + \lambda^*_{\texttt{R}}(t) \hat{\beta} \xr^*(t))(1 - x^*_\texttt{S}(t)+x^*_\texttt{R}(t)),
    \label{eq:phi_s}
   \\\phi_\texttt{V}(t) &=\langle\pmb{\lambda}^*(t),\textbf{g}_\texttt{V}(\textbf{z}^*(t))\rangle \nonumber
   \\&= (\lambda^*_\texttt{C}(t) c_\texttt{V} -\lambda^*_\texttt{S}(t) +\lambda^*_\texttt{R}(t)) \xs^*(t).
   \label{eq:phi_v}
\end{align}
The structure of the optimal control \um{inputs} $\up^*$ \um{and} $\uv^*$ are given as follows:
\begin{equation*}
    u^*_\texttt{P}(t) = 
    \begin{cases}
      u_{\texttt{Pmin}}, \quad & \text{if} \quad  \phi_\texttt{P}(t) > 0,  \\
       1, \quad & \text{if} \quad  \phi_\texttt{P}(t) < 0, \\
       \star, \quad & \text{if} \quad \phi_\texttt{P}(t) = 0,
    \end{cases}
\end{equation*}
\begin{equation}
    u^*_\texttt{V}(t) = 
    \begin{cases}
      u_{\texttt{Vmax}}, \quad & \text{if} \quad  \phi_\texttt{V}(t) < 0,  \\
       0, \quad & \text{if} \quad  \phi_\texttt{V}(t) > 0, \\
       \star, \quad & \text{if} \quad \phi_\texttt{V}(t) = 0,
    \end{cases}
    \label{eq:uv_phi_v}
\end{equation}
where $\star$ denotes the unknown candidate optimal control input, which is also referred to as a singular control input.

When $\phi_{\texttt{Q}} \nequiv 0$ (i.e., $\phi_{\texttt{Q}}$ is not identically zero on an open time-interval of $[0, T]\subset \mathbb{R}$) for $\texttt{Q} \in \{\texttt{P}, \texttt{V}\}$, then the optimal control inputs $\uv^*$ and $\up^*$ switch between their respective minimum and maximum admissible values, depending upon the sign of ${\phi}_{\texttt{Q}}$. The control inputs with this property are called bang-bang control inputs. However, we may also encounter a situation in which $\phi_{\texttt{Q}} \equiv 0$ is accompanied by the higher derivatives of $\phi_{\texttt{Q}}$ also vanishing on an open time-interval of $[0, T]\subset \mathbb{R}$, i.e., $\dot{\phi}_{\texttt{Q}} \equiv 0,~ \ddot{\phi}_{\texttt{Q}} \equiv 0,~ \dddot{\phi}_{\texttt{Q}} \equiv 0$, and so on. The control inputs which exhibit such a phenomenon are called singular control inputs (see, e.g., \cite{authour36}).

%%%%%%%%%%%%%%%%%%%%%%%%%%%%%%%%%%%%%%%%%%%%%%%%%%%%%%%%%%%%
%%%%%%%%%%%%%%%%%%%%%%%%%%%%%%%%%%%%%%%%%%%%%%%%%%%%%%%%%%%%
%%%%%%%%%%%%%%%%%%%%%%%%%%%%%%%%%%%%%%%%%%%%%%%%%%%%%%%%%%%%

\section{Non-existence of singular control inputs}\label{sec5}
%\subsection{Lie Bracket of Vector Fields}
An important mathematical tool required for the analysis of a singularity of an optimal control input is the Lie bracket. Let $\textbf{f}$ and $\textbf{g}$ be two continuously differentiable vector fields defined in $\mathbb{R}^n$. Then, for any given $\textbf{x} \in \mathbb{R}^n$, their Lie bracket is defined as follows:
\begin{equation}
    [\textbf{f}, \textbf{g}](\textbf{x})=D\textbf{g}(\textbf{x})\textbf{f}(\textbf{x})-D\textbf{f}(\textbf{x})\textbf{g}(\textbf{x}),
    \label{eq:lie_bracket}
\end{equation}
where $D\textbf{f}(\textbf{x})$ and $D\textbf{g}(\textbf{x})$ denote the Jacobian matrices of the vector fields $\textbf{f}$ and $\textbf{g}$, respectively, evaluated at the point $\textbf{x} \in \mathbb{R}^n$.

\subsection{Simultaneous non-singularity of the optimal control inputs $\up^*$ and $\uv^*$}
The existence of bang-bang control inputs (or equivalently non-existence of singular control inputs) is determined by the switching functions and their higher time-derivatives. As previously discussed, a singularity of \um{an optimal} control input arises when the switching function, $\phi_{\mathtt{Q}}$ \um{for $\texttt{Q} \in \{\texttt{P}, \texttt{V}\}$} vanishes identically over some time-interval, which is open in $[0, T]$. 

 First, we examine the possibility \um{of} both candidate \um{optimal} control inputs \um{being} simultaneously singular. It can be shown (see, e.g., \um{\cite[Section $4.4.3$]{authour35}}) that the switching functions $\phi_{\mathtt{P}}$ and $\phi_{\mathtt{V}}$ given by \eqref{eq:phi_s}
 and \eqref{eq:phi_v}, respectively, have higher order time-derivatives given by the following:
 \begin{align}
    \phi_{\mathtt{Q}} (t) &= \langle \pmb\lambda^*(t), \textbf{g}_{\mathtt{Q}}(\textbf{z}^*(t)) \rangle,
    \label{eq:phi_c_new}
    \\ \dot{\phi}_{\mathtt{Q}} (t) &= \langle \pmb\lambda^*(t), [\textbf{f}, \textbf{g}_{\mathtt{Q}}](\textbf{z}^*(t)) \rangle + \langle \pmb\lambda^*(t), [\gp, \textbf{g}_{\mathtt{Q}}](\textbf{z}^*(t)) \rangle \up^*(t)  + \langle \pmb\lambda^*(t), [\gv, \textbf{g}_{\mathtt{Q}}](\textbf{z}^*(t)) \rangle \uv^*(t),
    \label{eq:dphi_c_new}
    \\ \ddot{\phi}_{\mathtt{Q}} (t) &= \langle \pmb\lambda^*(t), [\textbf{f}, [\textbf{f} + \gp \up^* + \gv \uv^*, \textbf{g}_{\mathtt{Q}}]](\textbf{z}^*(t)) \rangle + \langle \pmb\lambda^*(t), [\gp, [\textbf{f} + \gp \up^* + \gv \uv^*, \textbf{g}_{\mathtt{Q}}]](\textbf{z}^*(t)) \rangle \up^*(t) \nonumber
    \\ & \quad+ \langle \pmb\lambda^*(t), [\gv, [\textbf{f} + \gp \up^* + \gv \uv^*, \textbf{g}_{\mathtt{Q}}]](\textbf{z}^*(t)) \rangle \uv^*(t),
    \label{eq:ddphi_c_new}
\end{align}
for $\mathtt{Q} \in \{\mathtt{P}, \mathtt{V}\}$. Before we state our result, we introduce the following assumption.
\begin{assumption}
    We proceed with the analysis on the class of diseases in which the reinfection rate exceeds the first time infection rate (i.e., $\hat{\beta} > \beta$).
    \label{assump:comp_imm}
\end{assumption}

The generalized SIRI epidemiological model discussed in this work takes reinfection into account. The finite-horizon optimal control problem presented in Eq \eqref{eq:ocp} involves several model parameters ($\beta, \hat{\beta}, \gamma$), costs ($c_{\mathtt{P}}, \cI, c_{\mathtt{V}}$), and two control inputs ($\up, \uv$) with defined lower and upper thresholds. Due to the large number of variables and parameters, it is quite challenging to derive a complete characterization for such a model. As a result, we focus on characterizing the control input behaviors for the case that the reinfection rate exceeds the initial infection rate satisfying $\hat{\beta} > \beta$. As mentioned above, during the recent COVID-19 pandemic, it was observed that the reinfection rates, particularly those associated with variants such as Delta and Omicron, were higher than the initial infection rates. The results of this paper are applicable to such immuno-compromising infectious diseases.

Now, we state the following proposition.
\begin{prop}
    \rev{Suppose Assumption \ref{assump:comp_imm} holds; then,} the candidate optimal control inputs $\uv^*$ and $\up^*$ \um{cannot} be simultaneously singular on \um{any} open time-interval $I \subset [0, T]$.
    \label{prop:simul_sing}
\end{prop}

\begin{proof}
As discussed in Section \ref{sec:structure_OC}, a control input exhibits a singularity when the switching function associated with it and its higher derivatives are all identically zero over an open time-interval. In our setting which is comprised of two control inputs, the necessary condition for the existence of a simultaneous singularity of the inputs \um{$\uv^*$ and $\up^*$ on $I$} requires the following:
    \begin{equation}
        \phi_{\mathtt{P}}(t) = \phi_{\mathtt{V}}(t) = \dot{\phi}_{\mathtt{V}}(t) = 0, \nonumber
    \end{equation}
    \um{to hold for every/almost every $t \in I$,} which implies the following:
    \begin{align}
        \langle \pmb\lambda^*(t), \textbf{g}_{\mathtt{P}}(\textbf{z}^*(t)) \rangle &= \langle \pmb\lambda^*(t), \textbf{g}_{\mathtt{V}}(\textbf{z}^*(t)) \rangle \nonumber\\
        &= \langle \pmb\lambda^*(t), [\textbf{f}, \textbf{g}_{\mathtt{V}}](\textbf{z}^*(t)) \rangle + \langle \lambda^*(t) [\gp, \textbf{g}_{\mathtt{V}}](\textbf{z}^*(t)) \rangle \up^*(t) \nonumber\\
        &= 0,
        \label{eq:simul_phi}
    \end{align}
    where we obtain \eqref{eq:simul_phi} from \eqref{eq:phi_c_new} and \eqref{eq:dphi_c_new}\um{, with} %we derive $\dot{\phi}_{\mathtt{V}} (t) = \langle \pmb\lambda^*(t), [\textbf{f}, \textbf{g}_{\mathtt{V}}](\textbf{z}^*(t)) \rangle + \langle \pmb\lambda^*(t), [\gp, \textbf{g}_{\mathtt{V}}](\textbf{z}^*(t)) \rangle \up^*(t)$, 
    \begin{equation}
    [\textbf{f}, \gv](\textbf{z}^*(t)) = \begin{bmatrix}
    \!\!\!\!\!\!&0\\
    \!\!\!\!\!\!&0\\
    \!\!\!\!\!\!&0
\end{bmatrix},\;\;\;
[\gp, \gv](\textbf{z}^*(t)) = \begin{bmatrix}
    \!\!\!\!\!\!\!\!& -\beta \cV \xs^*(t) (1 - \xs^*(t) - \xr^*(t))\\
    \!\!\!\!\!\!\!\!& 0\\
    \!\!\!\!\!\!\!\!& - \xs^*(t) (\beta - \hat{\beta}) (1 - \xs^*(t) - \xr^*(t)) 
\end{bmatrix}.
\label{eq:dphi_matrix_new}
\end{equation}
    It follows that \eqref{eq:simul_phi} holds if either of the two conditions is true: the co-state vector identically \um{vanishes (i.e., $\pmb\lambda^*(t) \equiv \textbf{0}$}) or the vectors $\gv(\textbf{z}^*(t))$, $\gp(\textbf{z}^*(t))$ and $[\textbf{f} + \gp \up^*, \gv](\textbf{z}^*(t))$ are linearly dependent \um{over $I$}. Now, from \eqref{transversality} and \eqref{lambda=0}, we conclude that $\lambda^*_{\texttt{C}}(t) \equiv 1$ is satisfied \um{for every $t \in I$}, which implies that the co-state vector $\pmb\lambda^*(t) \nequiv \textbf{0}$. Thus, the vectors $\gv(\textbf{z}^*(t))$, $\gp(\textbf{z}^*(t))$, and $[\textbf{f} + \gp \up^*, \gv](\textbf{z}^*(t))$ must be linearly dependent over $I$ for a simultaneous singularity of the inputs to exist. By computing the determinant of the matrix formed by these three vectors, we obtain the following (we have suppressed the explicit dependency of the states and control inputs on time for the sake of brevity):
    \begin{align}
        \Delta_{1}(\textbf{z}^*) &= \begin{vmatrix}
            \cV \xs^* \;\;\;\; & -\cP (\xs^* + \xr^*) \;\;\;\; & -\beta \cV \xs^* (1 - \xs^* - \xr^*) \up^*
            \\ -\xs^* \;\;\;\; & -\beta \xs^* (1 - \xs^* - \xr^*) \;\;\;\; & 0
            \\ \xs^* \;\;\;\; & -\hat{\beta} \xr^*(1 - \xs^* - \xr^*) \;\;\;\; & (\hat{\beta} - \beta) \xs^* (1 - \xs^* - \xr^*) \up^*
        \end{vmatrix} \nonumber
        \\ &= \begin{vmatrix}
            \cV \;\;\;\; & -\cP (\xs^* + \xr^*) \;\;\;\; & -\beta \cV
            \\ -1 \;\;\;\; & -\beta \xs^* (1 - \xs^* - \xr^*) \;\;\;\; & 0
            \\ 1 \;\;\;\; & -\hat{\beta} \xr^*(1 - \xs^* - \xr^*) \;\;\;\; & \hat{\beta} - \beta
        \end{vmatrix} \xs^{*{^2}} (1 - \xs^* - \xr^*) \up^*.
        \label{eq:det}
    \end{align}
    Suppose the three vectors are linearly dependent. \um{Observe from \eqref{eq:x=} that $\xs^{*}$ is strictly positive for a given bounded time-interval. Similarly, $\xI^{*} = 1 - \xs^{*} - \xr^{*}$ is also non-zero in the endemic case. In addition, $0 < u_{\mathtt{Pmin}} \leq \up^{*} \leq 1$ implies that the input $\up^{*}$ is strictly positive. Thus, $\xs^{*{^2}} (1 - \xs^* - \xr^*) \up^*$ is clearly non-zero on $I$.} Now, setting the determinant \um{in} \eqref{eq:det} equal to $0$ and using the relation $\xs^*(t) + \xr^*(t) \neq 0$ results in the following:
    \begin{equation}
        1 - \xs^* - \xr^* = \frac{\cP (\beta - \hat{\beta})}{\cV \beta \hat{\beta}}.
        \label{eq:simul_xi}
    \end{equation}
        %&\Delta_{1}(\textbf{z}^*) = \begin{vmatrix}
        %    \cV \;\;\;\; & -\cP (\xs^* + \xr^*) \;\;\;\; & -\beta \cV
        %    \\ -1 \;\;\;\; & -\beta \xs^* (1 - \xs^* - \xr^*) \;\;\;\; & 0
        %    \\ 1 \;\;\;\; & -\hat{\beta} \xr^* (1 - \xs^* - \xr^*) \;\;\;\; & \hat{\beta} - \beta
        %\end{vmatrix} = 0 \nonumber
        %\\  &\cV \beta \hat{\beta} \xr^* (1 - \xs^* - \xr^*) - (\beta - \hat{\beta}) \cP (\xs^* + \xr^*) + \cV \beta \hat{\beta} \xs^* (1 - \xs^* - \xr^*) = 0 \nonumber
    \um{Observe from \eqref{eq:x=} that $\xs^{*}(t)$ is an exponentially decreasing function, which remains strictly positive for a given bounded time-interval. Similarly, Lemma \ref{lem:invariant-set} ensures that $\xr^{*}(t) \geq 0$ holds. The left-hand side of \eqref{eq:simul_xi} corresponds to the state-variable $\xI^*(t)$, which resides in the set $[0, 1]$ by Lemma \ref{lem:invariant-set}. Whereas, the right-hand side is a negative constant under a compromised immunity (i.e., $\hat{\beta} > \beta$), thus implying that the equality \eqref{eq:simul_xi} can never hold.} Hence, the three vectors $\gv(\textbf{z}^*(t))$, $\gp(\textbf{z}^*(t))$, and $[\textbf{f} + \gp \up^*, \gv](\textbf{z}^*(t))$ are linearly independent. Thus, the control inputs $\uv^*(t)$ and $\up^*(t)$ can not be simultaneously singular on $I$. This concludes our proof.
\end{proof}

%%%%%%%%%%%%%%%%%%%%%%%%%%%%%%%%%%%%%%%%%%%%%%%%%%%%%%%%%%%%%%%%%
%%%%%%%%%%%%%%%%%%%%%%%%%%%%%%%%%%%%%%%%%%%%%%%%%%%%%%%%%%%%%%%%%
%%%%%%%%%%%%%%%%%%%%%%%%%%%%%%%%%%%%%%%%%%%%%%%%%%%%%%%%%%%%%%%%%

\subsection{Non-singularity of the optimal control input $\uv^*$}
First, we redefine \eqref{eq:ddphi_c_new} in terms of $\uv^*(t)$. By using the relation $[\textbf{f} + \gp \up^* + \gv \uv^*, \gv](\textbf{z}^*(t)) = [\gp, \gv](\textbf{z}^*(t)) \up^*(t)$ (since $[\textbf{f}, \gv](\textbf{z}^*(t)) = 0$ \um{and} $[\gv, \gv](\textbf{z}^*(t)) = 0$), we obtain the following:
\begin{align}
    \ddot{\phi}_{\mathtt{V}} (t) &= \langle \pmb\lambda^*((t), [\textbf{f}, [\gp, \gv]](\textbf{z}^*(t)) \rangle \up^*(t) + \langle \pmb\lambda^*(t), [\gp, [\gp, \gv]](\textbf{z}^*(t)) \rangle \up^{*^{2}}(t) \nonumber
    \\ &\quad + \langle \pmb\lambda^*(t), [\gv, [\gp, \gv]](\textbf{z}^*(t)) \rangle \up^*(t) \uv^*(t),
    \label{eq:ddot_phiv}
\end{align}
where
\begin{align*}
[\textbf{f}, [\gp, \gv]](\textbf{z}^*(t)) &= \begin{bmatrix}
    \!\!\!\!\!\!\!\!&\xs^*(t) (1 - \xs^*(t) - \xr^*(t)) ((\hat{\beta} - \beta) (\cI - \cP) + \beta \cV \gamma) \\
    \!\!\!\!\!\!\!\!&0\\
    \!\!\!\!\!\!\!\!&0
\end{bmatrix}, \;\;\;
    \\ [\gp, [\gp, \gv]](\textbf{z}^*(t)) &= \begin{bmatrix}
    \!\!\!\!\!\!\!\!&\xs^* (1 - \xs^*(t) - \xr^*(t)) ((\hat{\beta} - \beta) \cP + \beta^2 \cV (1 -\xr^*(t) - 2 \xs^*(t)) - \beta \hat{\beta} \cV \xr^*(t))\\
    \!\!\!\!\!\!\!\!&\beta \xs^{*^{2}}(t) (\beta - \hat{\beta}) (1 - \xs^*(t) - \xr^*(t))\\
    \!\!\!\!\!\!\!\!&\xs^*(t) (\beta- \hat{\beta}) (1 - \xs^*(t) - \xr^*(t)) ((\beta - \hat{\beta}) (1 - \xr^*(t)) + (\hat{\beta} - 2 \beta) \xs^*(t))
\end{bmatrix},
\\ [\gv, [\gp, \gv]](\textbf{z}^*(t)) &= \begin{bmatrix}
    \!\!\!\!\!\!\!\!&\beta \cV \xs^*(t) (1 - \xs^*(t) - \xr^*(t))\\
    \!\!\!\!\!\!\!\!&0\\
    \!\!\!\!\!\!\!\!&\xs^*(t) (\beta - \hat{\beta}) (1 - \xs^*(t) - \xr^*(t)) 
\end{bmatrix} 
\\&= -[\gp, \gv](\textbf{z}^*(t)).
\end{align*}
Before we state our main result, we state the following assumptions that will be essential for what follows. \rev{It is important to note that the weights $c_{\mathtt{P}}, c_{\mathtt{V}}$, and thresholds $u_{\mathtt{Pmin}}, u_{\mathtt{Vmax}}$ are parameters that policy makers are free to decide at the onset of a pandemic. We enforce the effective cost of protection $\cP (1 - u_{\mathtt{Pmin}})$ to be lower than the infection cost $\cI$ to incentivize the protection adoption. In addition, recall that we restrict our analysis to the class of immunocompromised diseases such that $\hat{\beta} > \beta$ holds. These assumptions motivate the following mathematical conditions in the form of Assumptions \ref{assump:beta_c_i}. A further discussion on the choice of parameters is included in Section \ref{sec:numerical}.}

\begin{assumption}
    We assume that the weighing and model parameters satisfy the following inequalities:
    \begin{enumerate}[(i)]
        \item $(\beta - \hat{\beta}) (\cI - \cP) \neq -\beta \cV \gamma$;
        \item $(\beta - \hat{\beta}) \cI + \beta \hat{\beta} \cV - \beta \cV \gamma < 0$; 
        \item $(\beta - \hat{\beta}) (\cI - \cP (1 - u_{\mathtt{{Pmin}}})) + \beta \hat{\beta} \cV u_{\mathtt{{Pmin}}} < 0$.
    \end{enumerate}
    \label{assump:beta_c_i}
\end{assumption}
%Observe that none of the above assumptions contradict the rationale typically applied during a pandemic. 
Assumptions \ref{assump:comp_imm} and \ref{assump:beta_c_i} are sufficient conditions under which Theorem \ref{theorem:uv_nonsing} holds.
\begin{theorem}
    Suppose Assumptions \ref{assump:comp_imm} and \ref{assump:beta_c_i} hold; then, the candidate optimal control input $\uv^*$ cannot be singular on \um{any} open time-interval \um{$I \subset [0, T]$}.
    \label{theorem:uv_nonsing}
\end{theorem}

\begin{proof}
    The proof is by contradiction. Assume that the control input \um{$\uv^*$} is singular \um{on $I$}. A singularity in \um{$\uv^*$} is obtained when the the switching function \um{$\phi_{\mathtt{V}}$} vanishes over $I$, which, in turn, implies $\phi_{\mathtt{V}}(t) = \dot{\phi}_{\mathtt{V}}(t) = \ddot{\phi}_{\mathtt{V}}(t) = 0$ for \um{every}/almost every $t \in I$. Equating $\dot{\phi}_{\mathtt{V}}(t)$ in \eqref{eq:simul_phi} to zero and leveraging the fact that $[\textbf{f}, \gv](\textbf{z}^*(t)) = \textbf{0}$\um{, implies} the following:
    \begin{equation}
        \langle \pmb\lambda^*(t), [\gp, \gv](\textbf{z}^*(t)) \rangle = 0.
        \label{eq:lamb_gpgv=0}
    \end{equation}
    As a result of \eqref{eq:lamb_gpgv=0}, the second time-derivative of the switching function in \eqref{eq:ddot_phiv} is given by the following: $$\ddot{\phi}_{\mathtt{V}}(t) = \langle \pmb\lambda^*(t), [\textbf{f}, [\gp, \gv]](\textbf{z}^*(t)) \rangle \up^*(t) + \langle \pmb\lambda^*(t), [\gp, [\gp, \gv]](\textbf{z}^*(t)) \rangle \up^{*^{2}}(t).$$ As the control input $\up^*(t) \neq 0$ (since $\up^*(t)$ lies in the interval $\up^*(t) \in [u_{\texttt{{Pmin}}}, 1]$, with $u_{\texttt{{Pmin}}} > 0$), on equating $\ddot{\phi}_{\mathtt{V}}(t) = 0$, we obtain the following:
    \begin{equation}
        \up^*(t) = - \frac{\langle \pmb\lambda^*(t), [\textbf{f}, [\gp, \gv]](\textbf{z}^*(t)) \rangle}{\langle \pmb\lambda^*(t), [\gp, [\gp, \gv]](\textbf{z}^*(t)) \rangle}.
        \label{eq:up_star}
    \end{equation}
    Now, we express the vector $[\gp, [\gp, \gv]](\textbf{z}^*(t))$ \um{as} follows:
    \begin{equation}
        [\gp, [\gp, \gv]](\textbf{z}^*(t)) = \epsilon(\textbf{z}^*(t)) \gv(\textbf{z}^*(t)) + \mu(\textbf{z}^*(t)) [\gp, \gv](\textbf{z}^*(t)) + \kappa(\textbf{z}^*(t)) [\textbf{f}, [\gp, \gv]](\textbf{z}^*(t)),
        \label{eq:ep_mu_kap}
    \end{equation}
    where $\epsilon$, $\mu$, $\kappa: \mathbb{R}^3 \rightarrow \mathbb{R}$ are given by (dependency on time has been dropped for clarity)
    \begin{align*}
        \epsilon &= \beta \xs^* (\hat{\beta} - \beta) (1 - \xs^* - \xr^*),
        \\ \mu &= (\hat{\beta} - \beta) (1 - \xs^* - \xr^*),
        \\ \kappa &= \frac{(\hat{\beta} - \beta) \cP + \beta \hat{\beta} \cV (1 - 2 \xs^* - 2 \xr^*)}{(\hat{\beta} - \beta) (\cI - \cP) + \beta \cV \gamma}.
    \end{align*}

\noindent It can be shown that the above obtained functions $\epsilon$, $\mu$, and $\kappa$ are indeed unique, since the three vectors \um{$\gv(\textbf{z}^*(t))$, $[\gp, \gv](\textbf{z}^*(t))$, and $[\textbf{f}, [\gp, \gv]](\textbf{z}^*(t))$ form a basis of $\mathbb{R}^{3}$ for each $t \in I$}. \rev{Note that under Assumption \ref{assump:beta_c_i}(i), when $(\beta - \hat{\beta}) (\cI - \cP) \neq -\beta \cV \gamma$ holds, $\kappa$ is well-defined.}

Rewriting the denominator of $\up^*(t)$ in \eqref{eq:up_star} (i.e., $\langle \pmb\lambda^*(t), [\gp, [\gp, \gv]](\textbf{z}^*(t)) \rangle$) in terms of the right-hand side of \eqref{eq:ep_mu_kap}, we obtain the following:
    \begin{align}
        \langle \pmb\lambda^*(t), [\gp, [\gp, \gv]](\textbf{z}^*(t)) \rangle &= \epsilon(\textbf{z}^*(t)) \underbrace{\langle \pmb\lambda^*(t), \gv(\textbf{z}^*(t)) \rangle}_\text{= 0 as $\phi_{\mathtt{V}}(t) = 0$} + \mu(\textbf{z}^*) \underbrace{\langle \pmb\lambda^*(t), [\gp, \gv](\textbf{z}^*(t)) \rangle}_\text{= 0 as $\dot{\phi}_{\mathtt{V}}(t) = 0$} \nonumber
        \\ & \quad+ \kappa(\textbf{z}^*(t)) \langle \pmb\lambda^*(t), [\textbf{f}, [\gp, \gv]](\textbf{z}^*(t)) \rangle.
        \label{eq:lamb_gpgpgv}
    \end{align}
    
\noindent Substituting \eqref{eq:lamb_gpgpgv} in \eqref{eq:up_star} under the \um{conditions} $\phi_{\mathtt{V}}(t) = \dot{\phi}_{\mathtt{V}}(t) = \ddot{\phi}_{\mathtt{V}}(t) = 0$ results in $\up^*(t) = -\frac{1}{\kappa(\textbf{z}^*(t))}$. Since, we assumed $\uv^*$ to be singular \um{on $I$}, as a consequence of Proposition \ref{prop:simul_sing}, it follows that $\up^*$ must exhibit a non-singularity \um{on} $I$. In other words, for a singularity of the optimal control input $\uv^*$ to exist\um{,} it is necessary that $\up^*$ is a bang-bang control (i.e., either $\up^*(t) = 1$, or $\up^*(t) = u_{\texttt{{Pmin}}}$ for \um{every}/almost every $t \in I$). \rev{When the control input $\up^*(t) = -\frac{1}{\kappa(\textbf{z}^*(t))} = 1$, the following holds:
     \vspace{-0.3cm}
    \begin{align}
         &\frac{(\beta - \hat{\beta}) \cI - \beta \cV \gamma}{\beta \hat{\beta} \cV} = 1 - 2 \xs^*(t) - 2 \xr^*(t) \nonumber
        \\ \implies &(\beta - \hat{\beta}) \cI - \beta \cV \gamma = \beta \hat{\beta} \cV \big(2 \xI^*(t) - 1 \big) \nonumber
        \\\implies &(\beta - \hat{\beta}) \cI - \beta \cV \gamma + \beta \hat{\beta} \cV = 2 \xI^*(t).
        \label{eq:up_contradict1}
    \end{align}
    Note that the right-hand side of \eqref{eq:up_contradict1} lies in the range of $[0, 2]$, whereas the left-hand side is strictly negative by Assumption \ref{assump:beta_c_i}(ii). Thus, \eqref{eq:up_contradict1} can not be true. Now, we consider the case when $\up^*(t) = -\frac{1}{\kappa(\textbf{z}^*(t))} = u_{\mathtt{{Pmin}}}$, where the following holds:
    \begin{align}
        &\frac{(\beta - \hat{\beta}) (\cI - \cP (1 - u_{\texttt{{Pmin}}}))}{\beta \hat{\beta} \cV u_{\mathtt{{Pmin}}}} = 1 - 2 \xs^*(t) - 2 \xr^*(t) \nonumber
       \\ \implies &(\beta - \hat{\beta}) (\cI - \cP (1 - u_{\mathtt{{Pmin}}})) = \beta \hat{\beta} \cV u_{\mathtt{{Pmin}}} \big(2 \xI^*(t) - 1 \big) \nonumber
       \\ \implies &(\beta - \hat{\beta}) (\cI - \cP (1 - u_{\mathtt{{Pmin}}})) + \beta \hat{\beta} \cV u_{\mathtt{{Pmin}}} = 2 \xI^*(t).
         \label{eq:up_contradict2}
    \end{align}
    Again, the right-hand side of \eqref{eq:up_contradict2} lies in the range of $[0, 2]$, whereas the left-hand side is strictly negative by Assumption \ref{assump:beta_c_i}(iii).} Thus, the structure of the \um{singular} control input $\up^*(t) = - \frac{\langle \pmb\lambda^*(t), [\textbf{f}, [\gp, \gv]](\textbf{z}^*(t)) \rangle}{\langle \pmb\lambda^*(t), [\gp, [\gp, \gv]](\textbf{z}^*(t)) \rangle}$ does not hold, which implies that $\ddot{\phi}_{\mathtt{V}}(t) = 0$ is not possible on $I$. By integrating $\ddot{\phi}_{\mathtt{V}}(t)$ twice, it is deduced that ${\phi}_{\mathtt{V}}(t) = 0$ on $I$ is also not possible. Thus, under Assumptions \ref{assump:comp_imm} and \ref{assump:beta_c_i}, the \um{optimal} control input $\uv^*$ is non-singular. This concludes our proof on the non-singular behavior of $\uv^*$.
\end{proof}
Note that the above result which characterizes the behavior of $\uv^*$ as a non-singular input is based on Assumptions \ref{assump:comp_imm} and \ref{assump:beta_c_i}. The impact of relaxing these assumptions presents an interesting research avenue, which we plan to explore in the future.
\subsection{Practical implication of theory}
Our theory demonstrated that there is no simultaneous singularity, nor is there any possibility of the vaccination input exhibiting a  singularity under the stated Assumptions \ref{assump:comp_imm} and \ref{assump:beta_c_i}. This has important implications on the public health policy. Specifically, the non-singularity results guarantee that the optimal vaccination policy is the simplest possible bang-bang control, which is often considered a more practical and appropriate intervention in epidemiological settings (see e.g., \cite{authour28, authour29}). The mathematical proof of Theorem \ref{theorem:uv_nonsing} rules out the existence of singularities in the vaccination input, which align with this broader understanding.

Further work needs to be done to fully characterize the optimal bang-bang control policy (i.e., determination of treatment level and transition times between treatment and no-treatment). Implementing our proposed strategies in real-world scenarios may be challenging, as implementation often requires adherence by individuals to the prescribed policies, and it is often difficult to ensure full compliance from people. It is to be noted that the optimality of a non-singular control has only been proven for the class of diseases leading to a compromised immunity (i.e., the reinfection rate is higher than the initial infection rate). As part of future work, we plan to expand our analysis to include diseases that confer a partial immunity, for which $\hat{\beta}$ is less than $\beta$. 

\section{Numerical results}\label{sec:numerical}

Now, we illustrate the trajectories of the optimal control inputs $\uv^*$ and $\up^*$, and the evolution of the SIRI dynamics through numerical simulations. \rev{We demonstrate different phenomena, through three different cases obtained by changing the parameters and costs. In the first two cases, we select the parameters and costs such that Assumptions \ref{assump:comp_imm} and \ref{assump:beta_c_i} are satisfied. In the third case, we violate the assumptions and illustrate the presence of singularities.} We choose \rev{$u_{\texttt{Pmin}} = 0.2$} and $u_{\texttt{Vmax}} = 0.9$. In addition, for the endemic equilibrium to exist, \um{the chosen parameters also satisfy the inequality $\gamma < \hat{\beta} u_{\texttt{Pmin}}$}. \rev{Accordingly, for the first two cases, we choose the following weighing and model parameters which satisfy the above mentioned assumptions, whereas for the third case, we choose $\hat{\beta} < \beta$,} with $\xs(0) = 0.8, \; \xI(0) = 0.2, \textnormal{ and } \xr(0) = 0$, where $x_{\texttt{j}}(0)$ for $\texttt{j} \in \{\texttt{S}, \texttt{I}, \texttt{R}\}$ denotes the initial state for the susceptible, infected, and recovered fractions of the population, respectively. The different costs and disease parameters are included in Table \ref{table1}. Note that a different set of parameters and costs will not violate the theoretical results proposed in Proposition~\ref{prop:simul_sing} and Theorem \ref{theorem:uv_nonsing}, as long as Assumptions \ref{assump:comp_imm} and \ref{assump:beta_c_i} hold. Thus, the main results remain robust to the choice of parameters, which are also illustrated in the numerical simulations. The choice of the reinfection and recovery rates are governed by the basic reproduction number, the commonly used as a metric in epidemiological studies, to determine the strength of an infectious disease. The disease spreads as long as the basic reproduction ratio is greater than one. Since we focus on the case of a compromised immunity, we choose an infection rate $\beta < \hat{\beta}$. Several studies emphasized the variability in the reproduction numbers of different COVID-19 viral strains. For example, the authors of \cite{authour37} predicted the reproduction number to be around $2.2$ in the early phase of COVID-19 spread, whereas the authors in \cite{authour38} estimated the Omicron reproduction number in certain parts of the world to be around $8$. Such a large variation is potentially due to the behavioral and environmental circumstances alongside viral mutations. We have chosen our infection and recovery rates which varies within a range of $R_{0} \in [1.32, 20]$. In the first case, we set the costs such that the protection cost, $c_{\mathtt{P}}$, dominates the other two costs, even though the effective cost of protection $\cP (1 - u_{\texttt{Pmin}})$ is lower than the infection cost. In the second and third cases, the protection cost is the lowest. The infection cost $c_{\mathtt{I}}$ being highest ensures that individuals are incentivized to choose either protection or vaccination, thus reducing the infection spread. The initial conditions of the model reflect that, a large proportion of the population is susceptible to the disease at the start of the epidemic, while only a small fraction is initially infected. We set the initial conditions such that there are no recovered individuals at the beginning of the epidemic, which corresponds to a first wave epidemic. 
\begin{table}[H]
\caption{Costs and disease parameters under various cases}
\label{table1}
\begin{center}
%\resizebox{9.5cm}{!}{
\begin{tabular}{lllllll} 
 \hline
  & $\cP$ & $\cV$ & $\cI$ & $\beta$ & $\hat{\beta}$ & $\gamma$ \\ [0.3ex] 
 \hline
 Case 1 & 7.1 & 2 & 7 & 1 & 2.5 & 0.38 \\ 
% \hline
 Case 2 & 0.3 & 2 & 5 & 1 & 2 & 0.1 \\ 
% \hline
 Case 3 & 0.3 & 3 & 5 & 3 & 2 & 0.1 \\ \hline
\end{tabular}
%}
\end{center}
\end{table}

Certain numerical methods exist to analyze the existence of singularities. Most of the existing methods are efficient on linear systems. Our system under study is non-linear, and such numerical methods require linearization of the system around the operating point. We numerically validate our analytical results using the numerical solver Quasi-Interpolation based Trajectory Optimization (QuITO), which is famous for solving constrained nonlinear optimal control problems (see \cite{authour39}). QuITO uses a direct multiple shooting (DMS) technique to discretize the control trajectory into several segments, and then obtains the optimal solution by solving for the control inputs at the boundaries of these segments. The trajectories which correspond to the states and control inputs for the three cases of weighing and model parameters are illustrated in Figure \ref{fig:only_case}.

The plots in the top row represent the control inputs, whereas the bottom panel illustrates the corresponding state trajectories. In the first two cases, when all assumptions are satisfied, we observe that the simultaneous singularity of control inputs, as well as 
the singularity in $\uv^*$, are completely absent, thus validating Proposition \ref{prop:simul_sing} and Theorem \ref{theorem:uv_nonsing}. 

In the first case (i.e., Figure \ref{fig:a}), when the protection cost is high, and when the infection prevalence becomes very low, we observe that the complete removal of protection is the optimal policy. This is illustrated by the switching in $\up^*$ from $\up^* = u_{\mathtt{Pmin}}$ to $\up^* = 1$ after 42 days. In the second case, when the protection cost is the cheapest, we observe a complete adoption of protection throughout the time-horizon, irrespective of the infection level. In the first two cases, an interesting observation is the behavior of the trajectory of the optimal control input $\uv^*$. At first, the behavior seems counter-intuitive; even though a sufficient fraction of the population is susceptible, vaccination is not applied as an input. This is explained based on the infection and reinfection rates $\beta$ and $\hat{\beta}$, respectively. Recall that the parameters are such that the rate at which the susceptible agents get infected (i.e., $\beta$) is lower than the rate at which the recovered agents get reinfected (i.e., $\hat{\beta}$). The susceptible agents have two options available to them: either (incurring a cost $\cV$) they transit to the recovered state ($\mathtt{R}$) by getting vaccinated where they are likely to get infected at a (comparatively higher) rate $\hat{\beta}$ or they remain in the susceptible state ($\mathtt{S}$) and get infected at a (comparatively lower) rate of $\beta$. Quite intuitively, the latter option seems optimal for the susceptible agents. In other words, the choice of applying vaccinations (and thereby incurring a vaccination weighing parameter), then transiting to the recovered state, and finally getting reinfected at a higher rate $\hat{\beta}$ is not optimal (note that the infection weighing parameter $\cI = 5$ dominates in the current scenario). Instead, agents prefer to remain in the susceptible state by not getting vaccinated. Therefore, we observe the optimal vaccination control $\uv^* \equiv 0$ being satisfied throughout the given time duration. This leads to the important observation that when the reinfection rate is higher than the initial infection rate in a SIRI model with a high infection cost, the susceptible individuals prefer to remain unvaccinated and get infected at a lower rate. 

Simulations in Figure \ref{fig:a} and \ref{fig:b} confirm the absence of simultaneous singularities for both the inputs $u_{\mathtt{P}}^*$ and $u_{\mathtt{V}}^*$, and the non- singularity of $u_{\mathtt{V}}^*$. These findings confirm the theoretical results outlined in the paper. It is important to note that a constant input represents a special case of non-singular control. Our analytical results focus on the special case of compromised immunity, where the reinfection rate ($\hat{\beta}$) exceeds the initial infection rate ($\beta$). As previously discussed, due to the high reinfection rate, individuals find it optimal not to get vaccinated, thus resulting in $u_{\mathtt{V}}^* = 0$, which is the lowest possible limit.
\begin{figure}[H]
     \centering
     \begin{subfigure}[b]{0.32\textwidth}
         \centering
         \includegraphics[width=\textwidth, trim={2cm 8cm 3cm 6cm},clip]{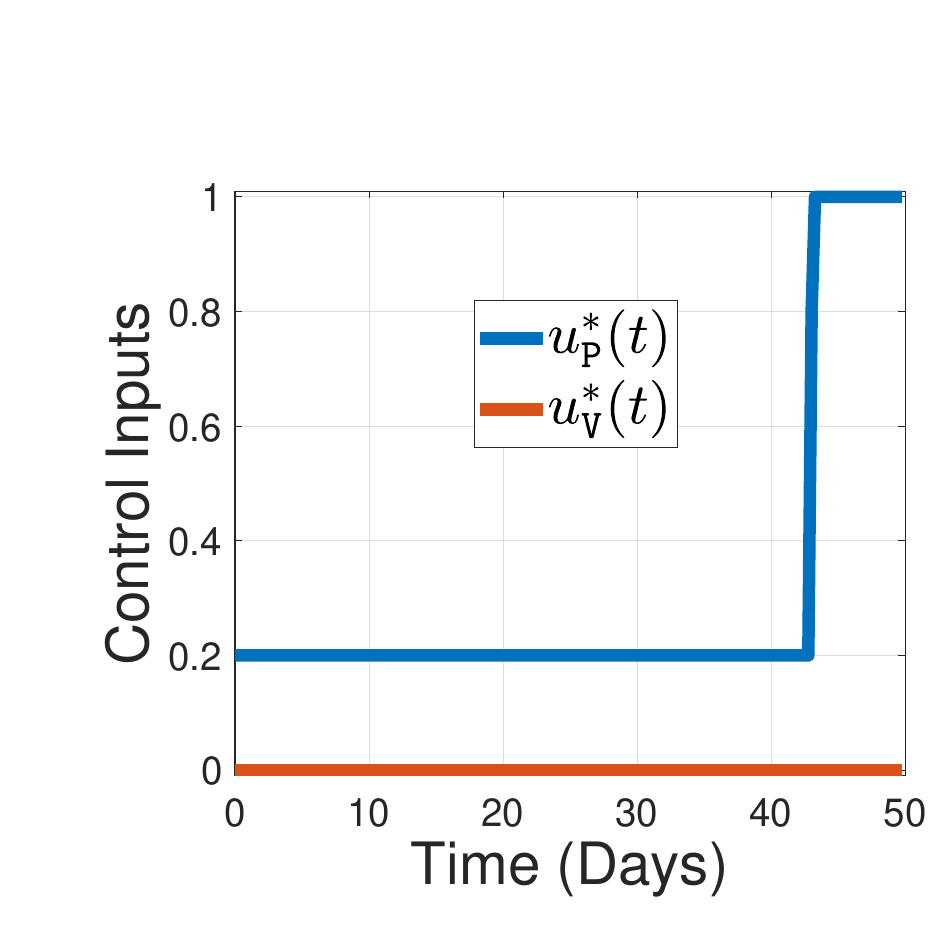}
         % \caption{States.}
         % \label{fig:a}
     \end{subfigure}
     \hfill
     \begin{subfigure}[b]{0.32\textwidth}
         \centering
         \includegraphics[width=\textwidth, trim={2cm 8cm 3cm 6cm},clip]{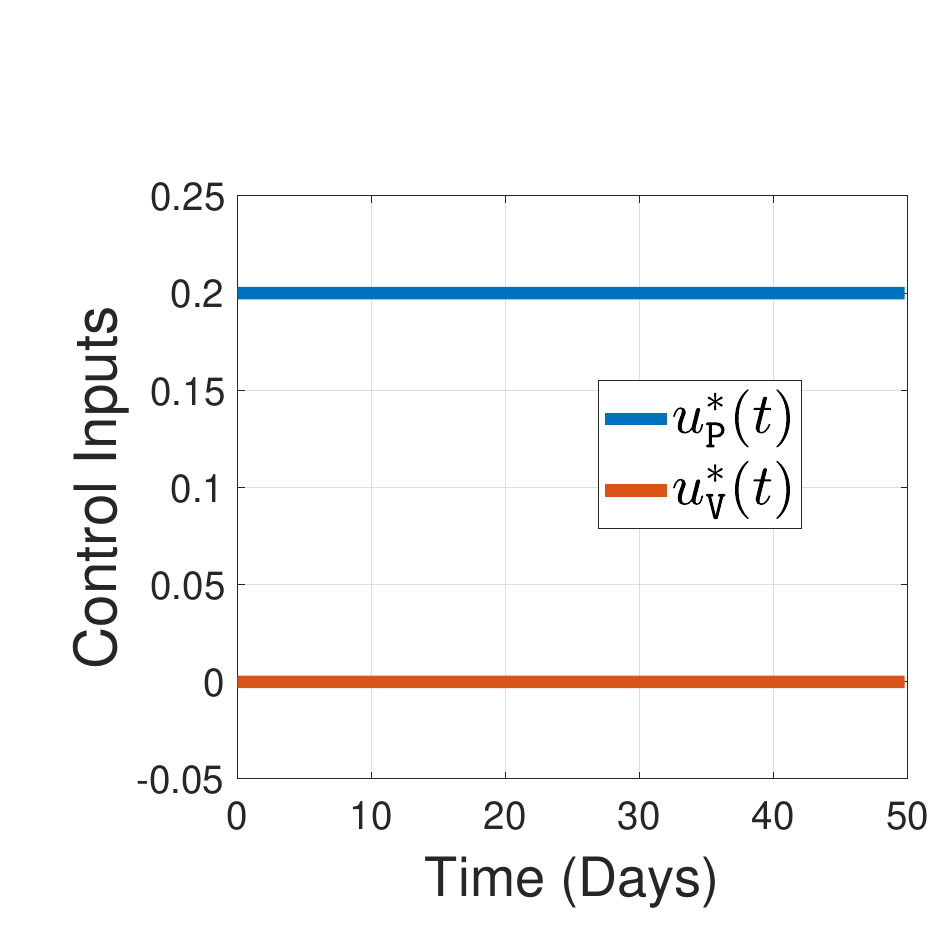}
         % \caption{Control Inputs.}
         % \label{fig:b}
     \end{subfigure}
     \hfill
     \begin{subfigure}[b]{0.32\textwidth}
         \centering
         \includegraphics[width=\textwidth, trim={2cm 8cm 3cm 6cm},clip]{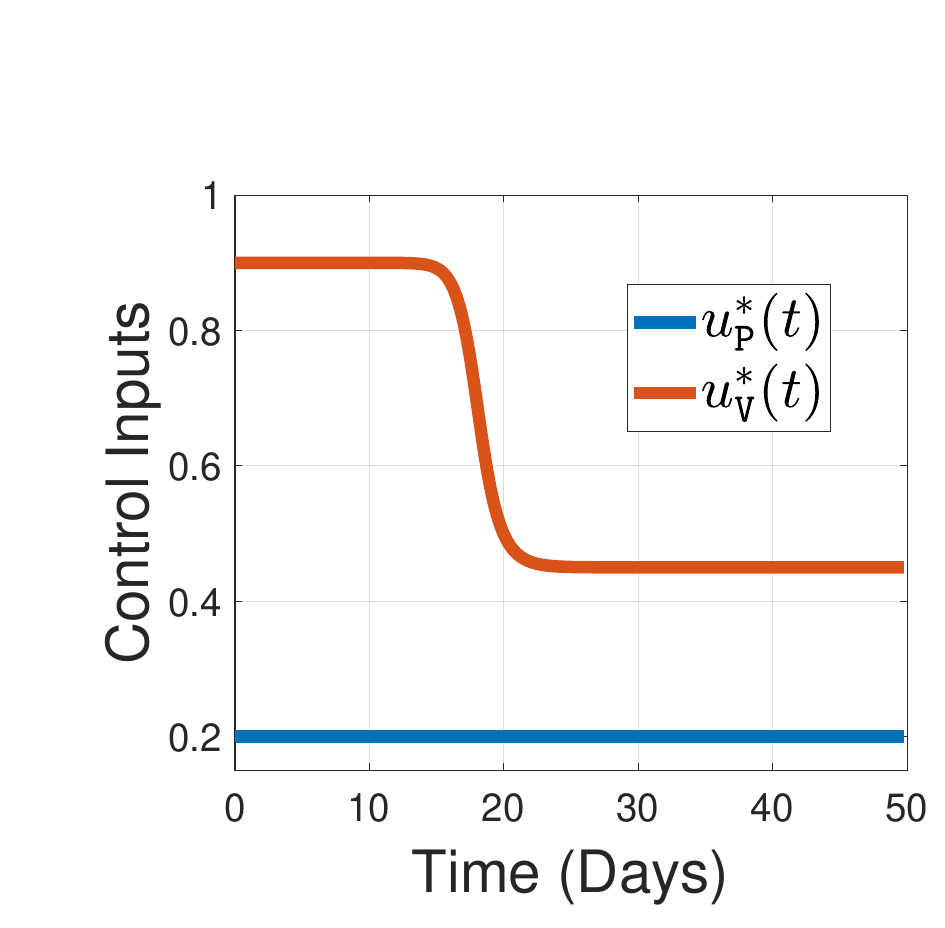}
         % \caption{Switching Functions.}
         % \label{fig:c}
     \end{subfigure}
     \begin{subfigure}[b]{0.32\textwidth}
         \centering
         \includegraphics[width=\textwidth, trim={2cm 8cm 3cm 9cm},clip]{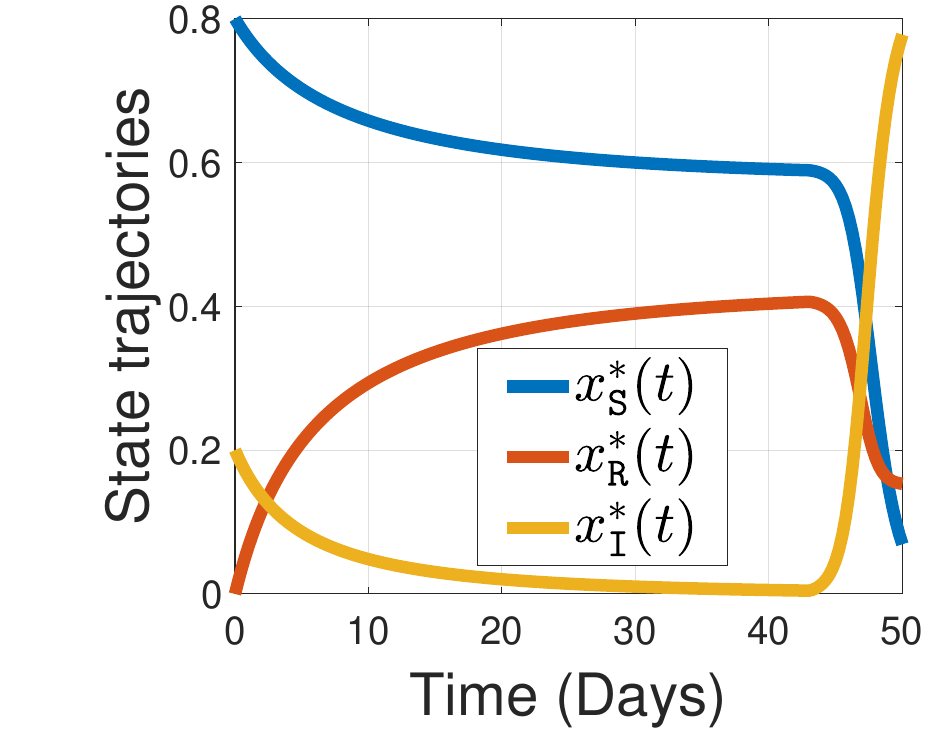}
         \caption{Switching $\up^*$}
         \label{fig:a}
     \end{subfigure}
     \hfill
     \begin{subfigure}[b]{0.32\textwidth}
         \centering
         \includegraphics[width=\textwidth, trim={2cm 8cm 3cm 9cm},clip]{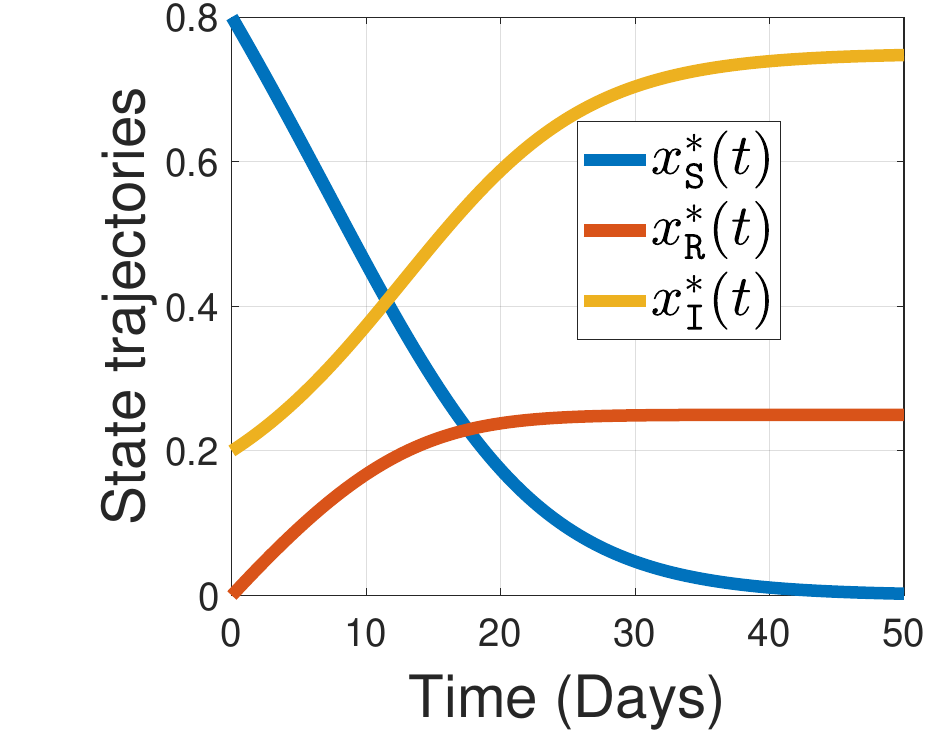}
         \caption{Constant inputs}
         \label{fig:b}
     \end{subfigure}
     \hfill
     \begin{subfigure}[b]{0.32\textwidth}
         \centering
         \includegraphics[width=\textwidth, trim={2cm 8cm 3cm 9cm},clip]{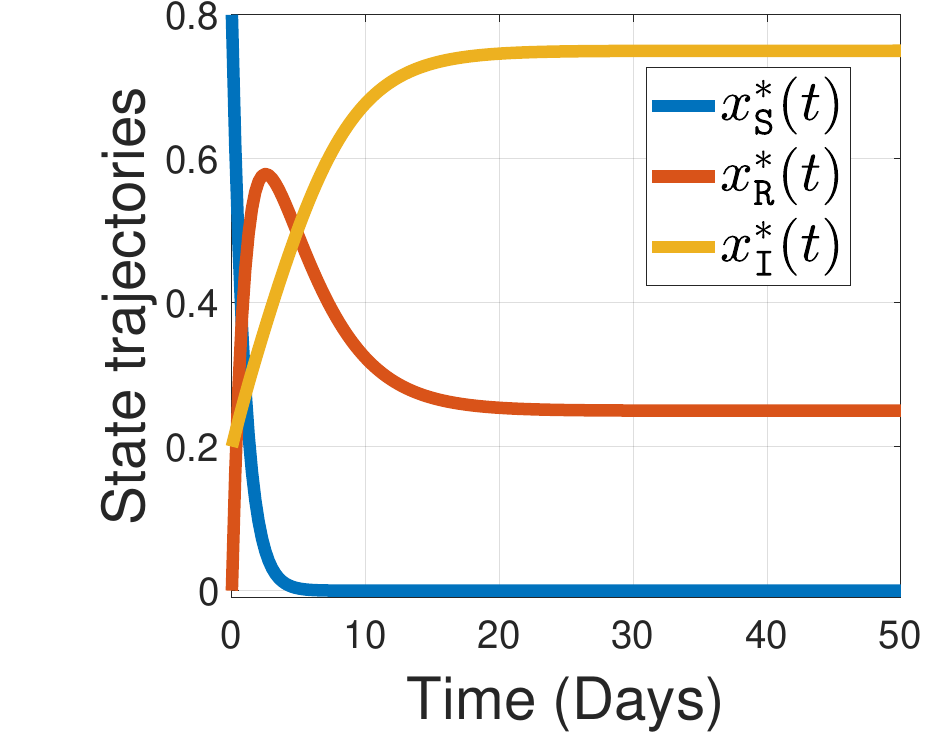}
         \caption{Singular $\uv^*$}
         \label{fig:c}
     \end{subfigure}
     \hfill
        \caption{Control inputs and state trajectories under different parameters.}
  \label{fig:only_case}
\end{figure}

 Now we focus on the simulations obtained under the third set of parameters. We violate Assumptions \ref{assump:comp_imm} and \ref{assump:beta_c_i} by selecting $\cV = 3$, $\beta = 3$, and $\beta > \hat{\beta}$, which implies partial immunity. Figure~\ref{fig:c} illustrates the smooth behavior of the optimal control, which implies a singular vaccination input $\uv^*$, whereas the input $\up^*$ remains constant at its lower limit. Additionally, this implies that Assumptions \ref{assump:comp_imm} and \ref{assump:beta_c_i} may be close to necessary for the existence of non-singular optimal control laws. A full analysis of the case in which $\beta > \hat{\beta}$ is an interesting area of research that remains to be explored in future work. The possible singularities in $\uv^*$ motivate us to investigate the behavior of diseases with a partial immunity.

Before concluding this section, we summarize the main advantages of our analysis and techniques in a broader context. We have demonstrated the non-existence of singular control inputs by analyzing the values of the time-varying switching functions $\phi_{\mathtt{P}}(t)$ and $\phi_{\mathtt{V}}(t)$. Necessary conditions for a singularity require these switching functions, along with their higher derivatives, to vanish identically. Our methodology is robust and can be applied to any compartmental model where the dynamics include control inputs and the running cost is linear in these inputs. Although our results focus on the relatively less-explored SIRI reinfection model, the technique for determining whether the control inputs are singular or non-singular could be useful for other epidemic models with different forms of control inputs. Thus, our technique is not model-sensitive and remains effective across various systems. This wide applicability of our analysis is due to the fundamental concepts of vanishing switching functions and their higher derivatives, which do not depend on the specific details of the underlying model.

\section{Conclusions and future work}\label{sec6}
In this paper, we considered the problem of optimal vaccination and protection for the class of SIRI epidemiological models. The biological significance of our results lies in the establishment of sufficient conditions on the susceptibility to infection and reinfection, and the cost of prevention and vaccination. Specifically, the proposed SIRI model takes reinfection into account, which is a key characteristic of diseases that result in short-term immunity. During the COVID-19 pandemic, we observed that the reinfection rates, particularly due to variants such as Delta and Omicron, were higher than the initial infection rates. Similarly, other diseases also exist which impart a compromised immunity. The existence of such real-world infectious diseases justifies the focus of this work on a compromised immunity. We proved that it is impossible for both the \um{optimal} control inputs to be simultaneously singular, when the immunity is compromised. Then, we performed a detailed analysis on the existence of a singularity of the optimal vaccination control input, and obtained sufficient conditions under which singular arcs (for optimal vaccination control input) are suboptimal and a non-singular vaccination control is optimal. Additionally, it is important to note that bang-bang control is often considered a more appropriate intervention in practical epidemiological settings. The numerical results provided valuable insights into the optimal control structure and evolution of the epidemic under such control inputs. Additionally, we illustrated that higher reinfection rates render vaccinations ineffective as the control input. 

Since our bang-bang control optimality results are only guaranteed in the regime of compromised immunity, an extension of the analysis for which the infection rate is higher than the reinfection rate (as also seen in Figure \ref{fig:c}) would be worthwhile. It will be worthwhile to extend our analysis to include diseases that impart a partial immunity to individuals. Establishing the existence of singularities, and deriving expressions of the singular controls remains as a future work. Furthermore, an empirical validation using real data (e.g., involving a controlled interventional challenge study) would be valuable. Another worthwhile extension of our work would be to include infection testing and contact tracing as additional control inputs when all states of the compartmental epidemiological model (e.g., the SAIRU model in \cite{authour40}) are not directly observable.

\section*{Author contributions}
Urmee Maitra: Conceptualization, Formal Analysis, Investigation, Methodology, Validation, Writing - original draft, review \& editing; Ashish R. Hota: Conceptualization, Formal Analysis, Writing - original draft, review \& editing; Rohit Gupta: Conceptualization, Formal Analysis, Writing - original draft, review \& editing; Alfred O. Hero: Conceptualization, Formal Analysis, Writing - original draft, review \& editing. 

\section*{Use of AI tools declaration}
The authors declare they have not used Artificial Intelligence (AI) tools in the creation of this article.

\section*{Acknowledgments (All sources of funding of the study must be disclosed)}
This research was partially supported by grant CCF-2246213 from the National Science Foundation. The authors thank anonymous reviewers for their helpful suggestions.

\section*{Conflict of interest}
The authors declare no conflict of interest in this paper.

\end{document}